\documentclass[10pt,a4]{amsart}
\usepackage{amssymb,mathrsfs,amsmath,dsfont,array}
\usepackage{tikz-cd}
\usepackage{hyperref}
\usepackage{mathtools}
\usepackage{makecell}

\usepackage{MnSymbol}%
\usepackage{wasysym}%

\newtheorem{The}{Theorem}[section]
\newtheorem{Lem}[The]{Lemma}
\newtheorem{Pro}[The]{Proposition}

\newtheorem{Cor}[The]{Corollary}

\theoremstyle{definition}

\theoremstyle{remark}
\newtheorem{remark}[The]{Remark}

\numberwithin{equation}{section}

\def\a1s{a_1,\cdots, a_s}
\def\a{\alpha}

\def\b{\beta}

\def\bl4{B_{\ell\geq4}}

\def\d{\delta}

\def\e{\epsilon}

\def\fk{\mathfrak{k}}

\def\lam{\lambda}

\def\fm{(\cdot,\cdot)}

\def\1k{\frac{1}{k}}
\def\op{\oplus}

\def\sub{\subseteq}

\def\bbbz{{\mathbb Z}}

\def\1il{1\leq i\leq\ell}

\renewcommand{\hat}{\widehat}

\newcommand{\summ}[1]{\raisebox{0.1ex}{\scalebox{1}{$\displaystyle \sum_{#1}\;$}}}
\newcommand{\Bigop}[1]{\raisebox{0.2ex}{\scalebox{.8}{$\displaystyle \bigoplus_{#1}\;$}}}

\begin{document}


\title{On results of certain modules over untwisted affine Lie superalgebras}

\author{Asghar Daneshvar$^*$}
\author{Hajar Kiamehr$^\star$}
\author{Maryam Yazdanifar$^\star$}
\author{Malihe Yousofzadeh$^{*,\star}$}

\address {$^{*}$School of Mathematics, Institute for Research in Fundamental Sciences (IPM), P.O. Box: 19395-5746, Tehran, Iran.}
\address{$^\star$Department of Pure Mathematics, Faculty of Mathematics and Statistics, University of Isfahan, Isfahan, P.O. Box 81746-73441, Iran.}

\email{a.daneshvar@ipm.ir}
\email{hkiamehr@sci.ui.ac.ir}
\email{yazdanifar.maryam@yahoo.com}
\email{ma.yousofzadeh@sci.ui.ac.ir, ma.yousofzadeh@ipm.ir.}



\begin{abstract}
Since 2020, finite weight modules have been studied over twisted affine Lie superalgebras. To complete the characterization of modules over affine Lie superalgebras, we need some information regarding modules over untwisted affine Lie superalgebras. There are several known results on representations of twisted affine Lie superalgebras that hold for untwisted cases, and their proofs are just a minor modification of the known ones.  In this note, we gather these results for our further use.
\end{abstract}

\maketitle


\newcommand\sfrac[2]{{#1/#2}}

\newcommand\cont{\operatorname{cont}}
\newcommand\diff{\operatorname{diff}}


\section{\bf Introduction}
In \cite{tight, Yousofzadeh:Finite weight 2020}, the author started the characterization of simple finite weight modules over twisted affine Lie superalgebras. She introduced two types of modules, hybrid and tight, depending on the action of root vectors corresponding to real roots.  In light of the characterization of hybrid modules that was done in \cite{Yousofzadeh:Finite weight 2020}, many results have been produced for twisted affine Lie superalgbras and modules over them. The main purpose of this note is to show that some of the observations made in \cite{Yousofzadeh:Finite weight 2020} can be expressed with minor modifications for the untwisted affine Lie superalgebras. 

Assume that $\mathfrak {L}$ is an affine Lie superalgebra
with the standard Cartan subalgebra $\frak H$. With respect to 
the canonical bilinear form on $\frak H^*$ (the dual space of $\mathfrak H$), we can divide the root system $R$ of $\mathfrak {L}$ into three kinds of roots: 
\begin{itemize}
    \item real roots: roots which are not self-orthogonal, 
    \item imaginary roots: roots which are orthogonal to all roots,
    \item nonsingular roots: neither real nor imaginary.
\end{itemize}
The action of root vectors corresponding to real roots on a simple finite weight $\frak L$-module $M$ has a determining role in studying the representation theory of affine Lie superalgebras. These actions on $M$ are performed in two types: injectively or locally nilpotently. Hence, we denote by $R^{in}$ (resp. $R^{ln}$) the set of all (nonzero) real root $\a$ whose corresponding nonzero root vectors act on $M$ injectively (resp. locally nilpotently).

In Section 3, we show that (see Theorem \ref{3.9}):

\begin{The}
	Suppose that $M$ is a simple finite weight module over untwisted affine Lie superalgebra  $\mathfrak L$. Then 
	for each $\beta \in R_{re}$, one of the following will happen:
	\begin{itemize}
		\item [(i)]  $\beta$ is full-locally nilpotent, equivalently $\beta +\mathbb Z\delta \subseteq R^{ln}$.
		\item [(ii)] $\beta$ is full-injective, equivalently $\beta +\mathbb Z\delta \subseteq R^{in}$.
		\item [(iii)] $\beta$ is down-nilpotent hybrid, equivalently there exist $m\in \mathbb Z$ and $t\in \{-1,0,1\}$ such that for $\gamma:=\beta +m\delta$, 
		\begin{align*}
			&\gamma +\mathbb Z^{\geq 1}\delta\subseteq R^{in},~~\gamma +\mathbb Z^{\leq 0}\delta\subseteq R^{ln}\\
			&-\gamma +\mathbb Z^{\geq t}\delta\subseteq R^{in},~~-\gamma +\mathbb Z^{\leq t-1}\delta\subseteq R^{ln}
		\end{align*}
		\item [(iv)]$\beta$ is up-nilpotent hybrid, equivalently there exist $m\in \mathbb Z$ and $t\in \{-1,0,1\}$ such that for $\eta:=\beta +m\delta$, 
		\begin{align*}
			&\eta +\mathbb Z^{\leq -1}\delta\subseteq R^{in},~~	\eta +\mathbb Z^{\geq 0}\delta\subseteq R^{ln}\\
			&-\eta +\mathbb Z^{\leq -t}\delta\subseteq R^{in},~~-\eta +\mathbb Z^{\geq 1-t}\delta\subseteq R^{ln}.
		\end{align*}
	\end{itemize}
\end{The}

\section{\bf Untwisted affine Lie superalgebras and their root systems}

In this short section, we present some of the basic definitions and notations regarding affine Lie superalgebras. In the whole of this note, we assume that the ground field is the complex numbers $\Bbb C$.

\subsection*{Lie superalgebras}

Consider a $\Bbb Z_2$-graded vector space $V=V_0\oplus V_1$ (super vector space) over $\Bbb C$. 
The nonzero elements of $V_0\cup V_1$ are called {\it homogeneous}. For a homogeneous element $v\in V_i$, we denote by $\bar v:=i$ (or ${\rm deg} (v):=i$) the degree of $v$. We allow below that if $\bar v$ is visible in any expression, then $v$ is supposed to be a homogeneous element.

A {\it Lie superalgebra} is a $\Bbb Z_2$-graded vector space $\frak g=\frak g_0\oplus \frak g_1$ over $\Bbb C$ whose Lie superbracket $[-, -]$ (a bilinear map from $\frak g \times \frak g$ to $\frak g$) satisfies the conditions:
\begin{itemize}
\item $[\frak g_i,\frak g_j]\subseteq \frak g_{i+j}$ for $i,j\in\Bbb Z_2$,
    \item Super skew-symmetry: 
    $[x,y]=-(-1)^{\bar x \bar y}[y,x]$,
    \item The super Jacobi identity:
 $  (-1)^{\bar x\bar z}[x,[y,z]]+(-1)^{\bar y \bar x}[y,[z,x]]+(-1)^{\bar z\bar y}[z,[x,y]]=0$.
\end{itemize}
A Lie superalgebra  $\frak g=\frak g_0\oplus \frak g_1$ is called {\it simple} if 
$\frak g_0\neq 0\neq \frak g_1$ and $[\frak g_0, \frak g_1]=\frak g_1$.
The simple Lie superalgebra $\frak g$ is said to be {\it classical simple} if $\frak g_0$ is reductive\footnote{Note that the even part ${\mathfrak {g}}_{0}$ is an ordinary Lie algebra.}. A classical simple Lie superalgebra $\frak g$ is called {\it basic} if $\frak g$ is equipped with an even nondegenerate invariant bilinear form\footnote{A bilinear form $(-,-) : \frak g \times \frak g\longrightarrow \Bbb C$ is called 
\begin{itemize}
    \item {\it invariant} if $([x,y],z)=(x,[y,z])$, for all $x,y,z\in \frak g$.
    \item {\it even} if $(\frak g_0, \frak g_1)=0$,
    \item {\it supersymetric} if $(x,y)=(-1)^{\bar x\bar y}(y,x)$.
\end{itemize}}. In \cite{Kac}, Victor Kac classifies all finite dimensional simple Lie superalgebras, and he shows that a finite dimensional basic classical simple Lie superalgebra (FDBCSL, for short) which is not a simple Lie algebra is isomorphic to one of the following:
\begin{align*}
    &\frak{sl}(m+1,n+1)~~ {\rm in~ which} ~~m>n\geq 0 ~~{\rm (Type~} A(m,n){\rm)},\\
    &\frak{psl}(n+1,n+1)~~ {\rm in~ which}~~n\geq 1 ~~{\rm (Type~} A(n,n){\rm)},\\
     &\frak{osp}(2m+1,2n)~~ {\rm in~ which}~~m\geq 0,~ n>0 ~~{\rm (Type~} B(m,n){\rm)},\\
      &\frak{osp}(2,2n-2)~~ {\rm in~ which}~~n\geq 2 ~~{\rm (Type~} C(n){\rm)},\\
       &\frak{osp}(2m,2n)~~ {\rm in~ which} ~~m\geq 2, n\geq 1 ~~{\rm (Type~} D(m,n){\rm)},\\
        &\Gamma(1,-1-\lam,\lam)~~ {\rm in~ which}~~ \lambda\neq 0,-1 ~~{\rm (Type~} D(2,1;\lam){\rm)},\\
        &G(3)~~ {\rm the~ exceptional~ Lie~ superalgebra~ of~ dimension} ~31,\\
        &F(4)~~{\rm the~ exceptional~ Lie~ superalgebra~ of~ dimension} ~40.
\end{align*}
We note that if $m=n$, then the elements $\Bbb C {\Bbb I}_{n+1,n+1}$, complex multiples of the identity matrix, are central in
 $\frak{sl}(n+1,n+1)$,  and so $\frak{psl}(n+1,n+1):= \frak{sl}(n+1,n+1)/ \Bbb C {\Bbb I}_{n+1,n+1}$ is simple.
The underlying even part of the finite dimensional basic classical simple Lie superalgebrs, Lie algebra, is presented in Table \ref{Table 1}.

\begin{table}[h!]
 \caption{The even part of FDBCSL}
\centering
 \begin{tabular}{||c|c ||c|c||} 
 \hline
 Type & $\frak g_{0}$ & Type & $\frak g_{0}$ \\ [0.5ex] 
 \hline\hline
$A(m,n)$,~
   ($m\neq n$) & $A_m\oplus A_n\oplus \Bbb C$& $D(m,n)$, ($m\neq 1$) & $D_m\oplus C_n$\\
    \hline
   $A(n,n)$ & $A_n\oplus A_n$&$F(4)$ & $A_1\oplus B_3$ \\
 \hline
$B(m,n)$ & $B_m\oplus C_n$&$G(3)$ & $A_1\oplus G_2$\\
 \hline
 $C(n)$ & $C_{n-1}\oplus \Bbb C$&$D(2,1;\lam)$ & $A_1\oplus A_1\oplus A_1$ \\
 \hline
 \end{tabular}
 \label{Table 1}
\end{table}

Let $\frak h_0$
be a subalgebra of $\frak g_0$ and denote its dual space by $ \mathfrak{h}^{\ast}$. Then, for $\a\in \frak h_0^*$, set
$$\frak {g}_\a:=\{x\in \frak g\mid [h,x]=\a(h)x~{\rm for~ all~}h\in \frak h_0^*\}.$$
We say $\frak g$ has a {\it root
space decomposition} with respect to $\frak h_0$ if
$$\frak g =\bigoplus _{\a \in \frak h_0^*}\frak g_\a,$$
and 
$$R:=\{\a\in \frak h_0^* \mid \a\neq 0,~ \frak g_\a \neq 0\}$$
is the set of {\it roots} of $\frak g$ ({\it root system of} $\frak g$)\footnote{This definition of root systems stems from semisimple Lie algebras.} An element of $\frak g_\a$ is called a {\it root vector} corresponding to $\a$.
For instance, if $\frak g$ is a finite dimensional basic classical simple Lie superalgebra
with the standard Cartan subalgebra $\frak h_0\subseteq \frak g_0$, then we can conclude that
$$\frak g =\frak g _0\oplus \bigoplus _{\a \in R}\frak g_\a,$$
 the root space decomposition, where $\frak g_0$ is the centralizer of $\frak h_0$ in $\frak g$.
 
 Suppose that $\frak g$ 
 is equipped with a nondegenerate even supersymmetric bilinear form $(-,-)$, $\frak h_0$ is finite
dimensional and the form $(-,-)$ is nondegenerate on $\frak h_0$.
Then, the form induces a supersymmetric nondegenerate bilinear form on $\frak h_0^*$ which is denoted again by $(-, -)$. Hence, if $\frak g$ has the root system $R$ and $S\subseteq R$, then we set
\begin{align*}
S_{im}&:=\{\a\in S\mid (\a,S)=\{0\}\}, \quad
 S^{\times} := S\setminus S_{im}, \\
 S_{re} &:= \lbrace \alpha \in S  \mid  (\alpha , \alpha) \neq 0 \rbrace, \qquad 
S_{ns} := \lbrace \alpha \in S^{\times} \mid   (\alpha , \alpha) = 0 \rbrace.
\end{align*}
This implies that $R$ is divided into three parts, $ R=R_{im} \cup R_{re} \cup R_{ns}$:
\begin{itemize}
\item {\it imaginary (isotropic) roots}, i.e., those roots which are orthogonal to all other roots,
\item nonzero {\it real roots}, i.e., those roots which are not self-orthogonal and
\item nonzero {\it nonsingular roots}, i.e., those roots which are neither real nor imaginary.
\end{itemize}
The root system satisfies some properties that lead us to the following:\\
Suppose that $V$ is a vector space with a nondegenerate supersymmetric bilinear form $(-,-)$ and $T\subseteq V$  such that 
\begin{itemize}
    \item [$\blacksquare$] ${\rm span}_\Bbb C(T)=V$ and $T$ is a finite set containing zero,
    \item [$\blacksquare$] $T=-T$,
    \item [$\blacksquare$] if $\a\in T_{re}$ and $\b\in T$, then the number 
$\langle \beta ,\alpha \rangle :=2{\frac {(\alpha ,\beta )}{(\alpha ,\alpha )}}$ is an integer.
    \item [$\blacksquare$] for any $\alpha \in T_{re}$ and $\beta \in T$, there are $p,q\in \Bbb Z^{\geq 0}$ such that $\langle \beta ,\alpha \rangle=p-q$ and 
    $$\{\b+k\a \mid k\in \Bbb Z\}\cap T=\{\b-p\a, \cdots , \b+q\a\},$$
    \item [$\blacksquare$] for $\a\in T_{ns}$ and $\b \in T$ with $(\a,\b)\neq 0$, 
    $\{\b-\a, \b+\a\}\cap T\neq \emptyset$.
   \end{itemize}
Any collection of vectors having these properties is called a {\it finite root supersystem}.
A finite root supersystem $R$ is said to be  {\it irreducible} if $R^\times$ cannot be written as a disjoint union of two nonempty orthogonal subsets.
More details on finite root super system, suppose that $ U $ is a complex vector space with a basis $ \lbrace \dot\gamma_{1}, \dot\gamma_{2}, \dot\gamma_{3} \rbrace. $ For $ \lambda \in \mathbb{C} \backslash  \lbrace 0, -1 \rbrace, $ define the symmetric nondegenerate bilinear form $ (- , -) $ on $ U $ by the linear extension 
\begin{equation}\label{2.0}
 (\dot\gamma_{1},\dot\gamma_{1}):= \lambda, ~ (\dot\gamma_{2},\dot\gamma_{2}):= -1 - \lambda, ~ (\dot\gamma_{3},\dot\gamma_{3}):= 1, ~ (\dot\gamma_{i},\dot\gamma_{j}):=  0 ~ ( 1 \leq i \neq j \leq 3 ). 
\end{equation}
Define  
\begin{equation}
\label{D(2,1;l)}
D(2, 1; \lambda ) := \lbrace 0, \pm 2 \dot\gamma_{i}, \pm \dot\gamma_{1} \pm \dot\gamma_{2} \pm \dot\gamma_{3} \mid 1 \leq i \leq 3 \rbrace. 
\end{equation}
Next, suppose that $ m $ and $ n $ are two nonnegative integers with $ m+n \geq 1 $ and $ F $ is a complex vector space with a basis  $ \lbrace \dot\epsilon_{i}, \dot\delta_{j} \mid 1 \leq i \leq m, 1 \leq j \leq n \rbrace. $ Define a symmetric bilinear form $ (-,-) : F \times F \longrightarrow \mathbb{C} $ with 
\begin{equation} \label{norm}
(\dot\epsilon_{i},\dot\epsilon_{r}):= \delta_{i,r}, ~~ (\dot\delta_{j},\dot\delta_{s}):=-\delta_{j,s} ~~{\rm and} ~~(\dot\epsilon_{i},\dot\delta_{j})=0,
\end{equation}
where $ 1 \leq i , r \leq m$, $1 \leq j, s \leq n$, and $\delta_{-,-}$ denotes the Kronecker delta. Set
\begin{align}
\label{s(m,n)}
 \frak{s} (m-1 , m-1) := \pm \lbrace \dot\epsilon_{i} - \dot\epsilon_{r}, \dot\delta_{i} - \dot\delta_{r}, \dot\epsilon_{i} - \dot\delta_{r} \mid 1 \leq i, r \leq m \rbrace ~ (n=m \in \mathbb{Z}^{\geq 2}) 
\end{align}
and
 \begin{align}\label{all}
 	\begin{split}
&A(m-1,m-1) := \pm \lbrace \dot\epsilon_{i} - \dot\epsilon_{r}, \dot\delta_{i} - \dot\delta_{r}, \dot\epsilon_{i} - \dot\delta_{r} - \dfrac{1}{m} \sum_{k=1} ^{m} (\dot\epsilon_{k} - \dot\delta_{k}) \mid 1 \leq i , r \leq m \rbrace ; (n=m \in \mathbb{Z}^{\geq 2}) \\ 
 &A(m-1,n-1) := \pm \lbrace \dot\epsilon_{i} - \dot\epsilon_{r}, \dot\delta_{j} - \dot\delta_{s}, \dot\epsilon_{i} - \dot\delta_{j} \mid 1 \leq i , r \leq m , 1 \leq j , s \leq n \rbrace ; ~ (n \neq m),  
\\ 
 &B(m,n) := \pm \lbrace \dot\epsilon_{i}, \dot\delta_{j}, \dot\epsilon_{i} \pm \dot\epsilon_{r}, \dot\delta_{j} \pm \delta_{s}, \dot\epsilon_{i} \pm \dot\delta_{j} \mid 1 \leq i,r \leq m, ~ 1 \leq j,s \leq n, i \neq r \rbrace, 
\\  
 &C(m,n) := \pm \lbrace \dot\epsilon_{i} \pm \dot\epsilon_{r}, \dot\delta_{j} \pm \dot\delta_{s}, \dot\epsilon_{i} \pm \dot\delta_{j} \mid 1 \leq i,r \leq m, 1 \leq j,s \leq n \rbrace,
 \\
 &D(m,n) := \pm \lbrace \dot\epsilon_{i} \pm \dot\epsilon_{r}, \dot\delta_{j} \pm \dot\delta_{s}, \dot\epsilon_{i} \pm \dot\delta_{j} \mid 1 \leq i,r \leq m, 1 \leq j,s \leq n,  i \neq r \rbrace,
 \\
 &BC(m,n) := \pm \lbrace \dot\epsilon_{i}, \dot\delta_{j}, \dot\epsilon_{i} \pm \dot\epsilon_{r}, \dot\delta_{j} \pm \dot\delta_{s}, \dot\epsilon_{i} \pm \dot\delta_{j} ~ \vert ~ 1 \leq i,r \leq m, ~ 1 \leq j,s \leq n \rbrace,
  \\
 &F(4) := \pm \lbrace 0, \dot\epsilon, \dot\delta_{i} \pm \dot\delta_{j}, \dot\delta_{i}, \frac{1}{2} (\dot\epsilon \pm \dot\delta_{1} \pm \dot\delta_{2}\pm \dot\delta_{3} ) ~ \vert ~ 1 \leq i \neq j \leq 3  \rbrace ;  (\dot\epsilon := \sqrt{3} \dot\epsilon_{1}, ~ m=1, ~ n=3)
 \\
 &G(3) := \pm \lbrace 0, \dot\nu , 2 \dot\nu , \dot\epsilon_{i} - \dot\epsilon_{j}, 2 \dot\epsilon_{i} - \dot\epsilon_{j} - \dot\epsilon_{t} , \dot\nu \pm (\dot\epsilon_{i} - \dot\epsilon_{j}) ~ \vert ~ \lbrace i,j,t \rbrace = \lbrace 1,2,3 \rbrace \rbrace ; (\nu := \sqrt{2} \delta_{1}, m=3, n=1),
\end{split}
 \end{align} 
in which if $ m $ or $ n $ is zero, the corresponding phrases disappear. The sets, introduced in (\ref{D(2,1;l)}) and (\ref{all}), are irreducible finite root supersystems in their linear spans. We draw the attention of readers to the point that $ \frak{s} (m-1 , m-1) $ is not an irreducible finite root supersystem, in fact, the form on the linear span of $ \frak{s} (m-1 , m-1) $ is not nondegenerate.
It is known that each irreducible finite root supersystem is either an irreducible finite root system\footnote{A subset $T$ of $\Bbb R^n$ is called a {\it finite root system} if
\begin{itemize}
    \item $T$ is finite, $0\in T$, and spans $\Bbb R^n$, 
    \item if $0\neq \a\in T$, then $-\a \in T$ and $\pm \a$ are the only multiples of $\a$ in $T$,
    \item for any two elements $\a, \b \in T$ with $\a\neq 0$, we have $r_\a(\b):=\b-{2\dfrac{(\b,\a)}{(\a,\a)}}\b\in T$ where $(-,-)$ is the form on $\Bbb R^n$,
    \item for any two elements $\a, \b \in T$ with $\a \neq 0$,  $<\b,\a>:={2\dfrac{(\b,\a)}{(\a,\a)}}$ is an integer.
\end{itemize}
A finite root system is said to be {\it irreducible} if it cannot be written as a disjoint union of two nonempty orthogonal subsets. The classification theorem of the finite root system can be found in \cite{Hum}.} or one of the irreducible finite root supersystems introduced in (\ref{D(2,1;l)}) and (\ref{all}).
We note that, in the literature, $ D(1,n) $ is denoted by $ C(n+1)$.
It should be noted that apart from types $BC(m,n)$ ($m+n\neq 0$) and $C(m,n)$ ($m,n>0$), the rest appear as
the root systems of basic classical simple Lie superalgebras.

\subsection*{Affine Lie superalgebras}
 Affine Lie superalgebras introduced by Johan van de Leur in 1986, see \cite{van-thes}. He constructed twisted and untwisted affine Lie superalgebras and showed that twisted and untwisted affine Lie superalgebras cover all affine Lie superalgebras.

 Affine Lie superalgebras are those contragradient Lie superalgebras corresponding to nonzero indecomposable symmetrizable generalized Cartan matrices of finite growth but not finite dimension. More precisely, suppose that $ n $ is a positive integer and $ I= \lbrace 1, \cdots , n \rbrace. $ Let $ \tau $ be a subset of $ I $ and $ A $ be a nonzero indecomposable symmetrizable $n \times n$ matrix with complex entries (in which by symmetrizable, we mean that $ A $ has a decomposition $ A =DB $ with an invertible diagonal matrix $ D $ and a symmetric matrix $ B $), satisfying the following:
\begin{itemize}
\item
if $ a_{i,j} = 0, \quad $ then  $ ~ a_{j,i} = 0; $
\item
if $ a_{i,i} = 0, \quad $ then $ ~ i \in \tau; $
\item
if $ a_{i,i} \neq 0, \quad $ then $ ~ a_{i,i} = 2; $
\item
if $ a_{i,i} \neq 0, \quad $ then $  a_{i,j} ~ ({\rm resp.}~ \dfrac{a_{i,j}}{2}) $ is a nonpositive integer for $ i \in I \setminus \tau ~ ({\rm resp.}~ i \in \tau ) $ with $ i \neq j. $
\end{itemize}
Fix a complex vector space $ \mathfrak{h} $ of dimension $ n +{\rm corank}(A) $ and denote its dual space with $ \mathfrak{h}^{\ast}. $ Then there exist linearly independent subsets
\begin{center}
$ \Pi := \lbrace \alpha_{i} ~ \vert ~ i \in I \rbrace \subseteq \mathfrak{h}^{\ast} \quad $ and $ \quad \check{\Pi} := \lbrace \check{\alpha}_{i} ~ \vert ~ i \in I \rbrace \subseteq \mathfrak{h} $
\end{center}
such that 
\begin{center}
$ \alpha_{j} ( \check{\alpha} _{i} ) = a_{i,j} \qquad (i,j \in I ). $
\end{center}
Let $ \tilde{\mathcal{G}} (A , \tau ) $ be the Lie superalgebra generated by $ \lbrace e_{i}, f_{i} ~ \vert ~ i \in I \rbrace \cup \mathfrak{H} $ subject to the following relations:
\begin{align*}
[e_{i},f_{j}] &=\delta_{i,j} \check{\alpha}_{i}, \qquad [h, h^{\prime} ] = 0,
\\
[h,e_{j}] &=\alpha_{j} (h) e_{j} , \quad [h,f_{j}] = - \alpha_{j} (h) f_{j},
\\
{\rm deg}(h) &=0, \qquad \qquad {\rm deg} (e_{i}) = {\rm deg} (f_{i}) =
\begin{cases}
0 \quad if ~ i \notin \tau
\\
1 \quad if ~ i \in \tau
\end{cases}
\end{align*}
for $ i,j \in I $ and $ h,h^{\prime} \in \mathfrak{h}$. According to \cite[Theorem 2.2.3(f)]{van-thes}, there is a unique maximal ideal $ \mathfrak{i} $ of $ \tilde{\mathcal{G}} (A , \tau ) $ intersecting $ \mathfrak{h} $ trivially. 
The (quotient) Lie superalgebra $ \mathcal{G} := \tilde{\mathcal{G}} (A , \tau )/ \mathfrak{i} $ is called an {\it affine Lie superalgebra} if it is not of finite dimension but of  finite growth\footnote{Note that the Lie superalgebra $\mathcal{G}$ has a $\mathbb Z$-grading $\mathcal{G}=\bigoplus_{i\in \mathbb Z} \mathcal{G}_i$, and the notion $\lim_{n\to\infty} \ln (\sum_{i=-n}^{n} \dim (\mathcal{G}_i))/\ln n$ is called the growth of $\mathcal{G}$.}. 

As for affine Lie algebras, affine Lie superalgebras are constructed using an affinizations process: \\
Suppose that $ \mathfrak{g} := \mathfrak{g}_{0} \oplus \mathfrak{g}_{1} $ ($\neq A(n,n)$) is a finite dimensional basic classical simple Lie superalgebra with a Cartan subalgebra $ \mathfrak{h}_0 \subseteq \mathfrak{g}_{0}. $ Suppose that $ \kappa $ is a nondegenerate supersymmetric invariant even bilinear form on $ \mathfrak{g} $ and $ \sigma $ is an automorphism of $ \mathfrak{g} $ of order $ k. $ Since $ \sigma $ preserves $ \mathfrak{g}_{0} $ as well as $ \mathfrak{g}_{1}. $ We have 
\begin{center}
$ \mathfrak{g}_{i} = \bigoplus \limits_{s=0} ^{k-1}{}^{[s]} \mathfrak{g}_{i} \quad $ where $ \quad ^{[s]} \mathfrak{g}_{i} = \lbrace x \in \mathfrak{g}_{i} ~ \vert ~ \sigma (x) = \varsigma^{s} x \rbrace  \qquad (i \in \mathbb{Z}_{2} , 0 \leq s \leq k-1) $
\end{center}
in which $ \varsigma $ is the $ k$-th primitive root of unity. Then 
\begin{center}
$ \hat{\mathfrak{g}} := \hat{\mathfrak{g}_{0}} \oplus \hat{\mathfrak{g}_{1}} \quad $ where $  \quad \hat{\mathfrak{g}_{i}} = \bigoplus \limits_{s=0} ^{k-1} ( ^{[s]} \mathfrak{g}_{i} \otimes t^{s} \mathbb{C} [t^{\pm k}])  \qquad (i \in \mathbb{Z}_{2}) $
\end{center}
is a subalgebra of the current superalgebra $ \mathfrak{g} \otimes \mathbb{C} [t ^{\pm 1}]. $ Set
\begin{center}
$ \mathscr{G} = \bigoplus \limits_{s=0} ^{k-1} ( ^{[s]} \mathfrak{g} \otimes t^{s} \mathbb{C} [t^{\pm k}]) \oplus \mathbb{C} c \oplus \mathbb{C} d ~ $ and $ ~ \frak{h} := ((^{[0]} \mathfrak{g} \cap \mathfrak{h}_0 ) \otimes 1 ) \oplus \mathbb{C} c \oplus \mathbb{C} d,  $
\end{center}
then $ \mathscr{G} $ together with
\begin{center}
$ [x \otimes t^{p} + rc +sd , y \otimes t^{q} +
r^{\prime} c + s^{\prime} d ] := [x,y] \otimes t^{p+q} + p \kappa (x,y) \delta_{p + q , 0} c+ sqy \otimes t^{q} - s^{\prime} px \otimes t^{p} $
\end{center}
is an {\it affine Lie superalgebra} and $ \frak{h} $ is a Cartan subalgebra of $ \mathscr{G}. $ The Lie superalgebra $ \mathscr{G} $ is denoted by $ X^{(k)} $ where $ X $ is the type of $ \mathfrak{g} $ and $ X \neq A(n,n). $
 The lie superalgebras $ X^{(k)} $ is called {\it twisted} if $ k > 1 $ and it is called {\it untwisted} (or {\it non-twisted}) if $ k=1$. 

 For $ X = A(n,n)$, the affinizations process is a little different:\\
 Let ${\frak g}= A(n,n)$. Then we can assume that   
$${\frak g}={\frak g}_0\oplus {\frak g}_1=\frak{g}_-\oplus \frak{g}_0\oplus \frak{g}_+$$
in which
\begin{align*}
&\frak{g}_0=\left\{\begin{pmatrix}
A& O\\
O&D
\end{pmatrix}~|~A, D\in {\rm Mat}_{n+1,n+1}(\Bbb C),~{\rm tr} A={\rm tr} D\right\},\\
&\frak{g}_1=\left\{\begin{pmatrix}
O& B\\
C&O
\end{pmatrix}~|~B,C\in {\rm Mat}_{n+1,n+1}(\Bbb C)\right\},\\
&\frak{g}_-=\left\{\begin{pmatrix}
O& O\\
C&O
\end{pmatrix}~|~C\in {\rm Mat}_{n+1,n+1}(\Bbb C)\right\},\\
&\frak{g}_+=\left\{\begin{pmatrix}
O& B\\
O&O
\end{pmatrix}~|~B\in {\rm Mat}_{n+1,n+1}(\Bbb C)\right\}.
\end{align*}
Define
$$
\delta^*: {\frak g}\longrightarrow {\frak g}, ~X\mapsto \left\{ \begin {array} {cc}  X&~~ X\in \frak{g}_+ \\ 0 &~~ X\in \frak{g}_0 \\ - X &~ X\in \frak{g}_-
\end {array} \right.$$

\noindent and 
$\hat{{\frak g}}:={\frak g}\otimes \Bbb C[t,t^{-1}]$. 
Set 
\begin{align*}
&\omega : \hat{{\frak g}}\times \hat{{\frak g}} \longrightarrow \Bbb C c_1\oplus \Bbb Cc_2\\
&(X\otimes t^m, Y\otimes t^u)\mapsto \frac{1}{n+1}\delta_{m+u,0}(\delta^*(X),Y)c_1+m\delta_{m+u,0}(X,Y)c_2.
\end{align*}
and
$\bar {\frak g}:=\hat{{\frak g}}\oplus \Bbb Cc_1 \oplus \Bbb Cc_2 $.
With the following  rule 
$$[X\otimes t^m,Y\otimes t^u]_\omega:=[X,Y]\otimes t^{m+u}+\omega(X\otimes t^m,Y\otimes t^u),$$
$\bar {\frak g}$ is a Lie superalgebra. Note that for 
$x,y,z\in \hat{{\frak g}}$, we have
$$\omega ([x,y],z) +\omega ([z,x],y)+ \omega ([y,z],x)=0.$$
Now, consider
$$\frak {L}:=\bar {\frak g}\oplus \Bbb Cd_1\oplus \Bbb Cd_2=
{\frak g}\otimes \Bbb C[t,t^{-1}]\oplus \Bbb Cc_1 \oplus \Bbb Cc_2\oplus\Bbb Cd_1\oplus \Bbb Cd_2.$$
With the following relations, $\frak {L}$ is a Lie superalgebra
\begin{align*}
&[\frak {L},\Bbb Cd_1\oplus \Bbb Cd_2]=0,~~~~~~~~~~~[d_1,d_2]=0\\
&[d_1,X\otimes t^m]=2\delta^*(X)\otimes t^m, ~~ X\in {\frak g}\\
&[d_2,X\otimes t^m]=mX\otimes t^m, ~~ X\in {\frak g}\\
&[x,y]=[x,y]_\omega,~~~~~~~~~~~~~(x,y\in \bar {\frak g})
\end{align*}
and it is called the {\it nuntwisted affine Lie superalgebra of type} $A(n,n)^{(1)}$.
As we mentioned, twisted and untwisted affine Lie superalgebras cover all affine Lie superalgebras.

Next, we introduce the root system of untwisted affine Lie superalgebras\footnote{Refer to \cite{Yousofzadeh:Finite weight 2020} for the details of the root system of twisted affine Lie superalgeras.}:
\subsection*{Types $X\neq A(n,n)^{(1)}$}
Define $ \delta $ to be the functional on $\frak{h}$ vanishing on $ ((^{[0]} \mathfrak{g} \cap \mathfrak{h} ) \otimes 1) \oplus \mathbb{C} c $ and mapping $ d $ to $ 1. $ If $ \dot{R} $ is the root system of the finite dimensional basic classical simple Lie superalgebra $ \mathfrak{g}, $ introduced in (\ref{D(2,1;l)}) and (\ref{all}), $R:=\dot{R} + \mathbb{Z} \delta $ is the root system of $ \mathscr{G} = X^{(1)}. $

Using the form on the finite dimensional basic classical simple Lie superalgebra  $ \mathfrak{g}$ (see (\ref{2.0}) and (\ref{norm})), we can get a nondegenerate supersymmetric invariant even bilinear form on $ \mathscr{G} $ which, in turn, induces a symmetric bilinear form $ (-,-) $ on  $ \frak{h}^* $ with the following expansion:
\begin{center}
$  (\delta,\mathscr{G}):= \lbrace 0 \rbrace, \quad (\dot\epsilon_{i},\dot\epsilon_{j}):= \delta_{i,j}, \quad (\dot\delta_{p},\dot\delta_{q}):=-\delta_{p,q}, \quad (\dot\epsilon_{i},\dot\delta_{p}):= 0 .$
\end{center}
Table \ref{Table 2} demonstrates all kinds of roots of $ \mathscr{G} $.

\begin{table}[h!]
 \caption{Kinds of roots for type $X\neq A(n,n)^{(1)}$}
\centering
 \begin{tabular}{||c| c| c| c||} 
 \hline
 $X^{(1)}$ & $R_{im}$ & $R_{re}$ & $R_{ns}$ \\ [0.5ex] 
 \hline\hline
\makecell{$A(m-1,n-1)^{(1)}$,\\
   ($m\neq n$)} &$\Bbb Z \d$ & \makecell{$\{\dot\e_i-\dot\e_r,\dot\d_j-\dot\d_s \mid 1\leq i\neq r\leq 
   m,$\\ $1\leq j\neq s\leq n \}+\Bbb Z \d$}& \makecell{$\pm\{\dot\e_i-\dot\d_j \mid 1\leq i\leq m,~1\leq j\leq n\}$\\$+\Bbb Z \d $}\\
 \hline
$B(m,n)^{(1)}$ & $\Bbb Z \d$ & \makecell{$\pm \{\dot\epsilon_{i}, \dot\delta_{j}, 
\dot\epsilon_{i} \pm \dot\epsilon_{r}, \dot\delta_{j} \pm \dot\delta_{s},2\dot\d_j \mid 1 \leq i\neq r \leq 
m,$\\ $~ 1 \leq j\neq s \leq n\}+\Bbb Z\d$}&  \makecell{$\pm \{  \dot\epsilon_{i} \pm\dot\delta_{j} \mid 1 \leq i \leq m, ~ 1 \leq j \leq n\}$\\
$+\Bbb Z\d$}\\
 \hline
 \makecell{$C(n)^{(1)}=$\\
 $D(1,n-1)^{(1)}$} & $\Bbb Z \d$ & \makecell{$\pm \{ 2\dot\epsilon_1 ,
 \dot\delta_{j} \pm \dot\delta_{s},2\dot\d_j\mid  1 \leq j\neq s \leq n-1\}$\\$+\Bbb Z\d$} & \makecell{$\pm \{ \dot\epsilon_1 \pm \dot\delta_{j}\mid  1 \leq j \leq n\}$\\$+\Bbb Z\d$} \\
 \hline
 \makecell{$D(m,n)^{(1)}$\\
 $(m\neq 1)$} & $\Bbb Z \d$ & \makecell{$\pm \{ \dot\epsilon_{i} \pm \dot\epsilon_{r},
 \dot\delta_{j} \pm \dot\delta_{s},2\dot\d_j\mid 1 \leq i\neq r \leq m,$\\$ 1 \leq j\neq s \leq n
 \}+\Bbb Z\d$} & \makecell{$\pm \{ \dot\epsilon_{i} \pm \dot\delta_{j}\mid 1 \leq i \leq 
 m, 1 \leq j \leq n\}$\\$+\Bbb Z\d$} \\
 \hline
 $F(4)^{(1)}$ & $\Bbb Z \d$ & \makecell{$\pm \{\dot\epsilon, \dot\delta_{i} \pm 
 \dot\delta_{j}, \dot\delta_{i}\mid ~ 1 \leq i \neq j \leq 3  \}$\\$+\Bbb Z\d$} & 
 \makecell{$\pm \{\frac{1}{2} (\dot\epsilon \pm \dot\delta_{1} \pm \dot\delta_{2}\pm 
 \dot\delta_{3} )\}+\Bbb Z\d$} \\
 \hline
 $G(3)^{(1)}$ & $\Bbb Z \d$ & \makecell{$\pm \{\dot\nu , 2 \dot\nu , \dot\epsilon_{i} - \dot\epsilon_{j}, \dot\nu \pm (\dot\epsilon_{i} - \dot\epsilon_{j}) \mid i\neq j$\\$\in \{ 1,2,3 \} \}+\Bbb Z\d$} & \makecell{$\pm \{ 2 \dot\epsilon_{i} - \dot\epsilon_{j} - \dot\epsilon_{t}\mid \{ i,j,t \} = \{ 1,2,3 \} \}$\\$+\Bbb Z\d$} \\
 \hline
 $D(2,1;\lam)^{(1)}$ & $\Bbb Z \d$ & $\{\pm 2\dot\gamma_i \mid i\in \{ 1,2,3 \} 
 \}+\Bbb Z\d $& $\{\pm\dot\gamma_1\pm\dot\gamma_2\pm\dot\gamma_3\} +\Bbb Z\d$\\
 \hline
 \end{tabular}
 \label{Table 2}
\end{table}

\subsection*{Type $X= A(n,n)^{(1)}$} The story of the root system for type $X= A(n,n)^{(1)}$ needs more details. Let $E_{i,j}\in \frak g=A(n,n)$
denote the matrix where the entry in the i-th row and the j-th column is 1. Set
 $$H_i=E_{i,i}-E_{i+1,i+1}, ~~L_i=E_{i+n+1,i+n+1}-E_{i+l+2,i+n+2} \quad (1\leq i\leq n)$$
 and 
 $$\frak h_0:={\rm span} _{\Bbb C}\{H_i,L_i~|~1\leq i\leq n\},$$
$\frak h_0$ is the standard Cartan subalgebra of ${\frak g}_0$. Define
$$
\begin{cases}
\dot\epsilon_i: {\frak h_0}\longrightarrow \Bbb C\\
\indent H_j\mapsto \delta_{i,j}-\delta_{i,j+1}\\
\indent L_j\mapsto 0
\end{cases}
\quad\quad \begin{cases}
\dot\delta_i: {\frak h_0}\longrightarrow \Bbb C\\
\indent L_j\mapsto \delta_{i,j}-\delta_{i,j+1}\\
\indent H_j\mapsto 0
\end{cases}
$$
where  $1\leq i\leq n+1$ and
$1\leq j\leq n$. 
The root system of $\frak g$ is
$$
\dot{R}=\{\dot{\epsilon}_i-\dot{\epsilon}_j,\dot{\delta}_i-\dot{\delta}_j,\dot{\epsilon}_i-\dot{\delta}_j,\dot{\delta}_j-\dot{\epsilon}_i~|~ 1\leq i,j\leq n+1\}.$$
 Set
$${\frak h}:={\frak h_0}\otimes \Bbb C \oplus \Bbb C c_1\oplus \Bbb C c_2\oplus \Bbb Cd_1\oplus \Bbb Cd_2.$$ 
 Now, suppose that
$\dot{\alpha} \in \dot{R}$. Then, define
$$\dot{\alpha}(c_1)=\dot{\alpha}(c_2)=\dot{\alpha}(d_1)=\dot{\alpha}(d_2)=0.$$
Hence, we can consider $\dot{\alpha}$ as a functional on 
$\frak h$. Next, define
$$
\begin{cases}
\sigma: {\frak h}\longrightarrow \Bbb C\\
\indent h,c_1,c_2,d_2\mapsto 0\\
\indent d_1\mapsto 2
\end{cases}
\quad\quad \begin{cases}
\delta: {\frak h}\longrightarrow \Bbb C\\
\indent h,c_1,c_2,d_1\mapsto 0\\
\indent d_2\mapsto 1
\end{cases}
$$
where
 $h\in {\frak h}$. The root system of $A(n,n)^{(1)}$ with respect to $\frak h$ is
$$R=\{\dot{\epsilon}_i-\dot{\epsilon}_j,\dot{\delta}_i-\dot{\delta}_j,\dot{\epsilon}_i-\dot{\delta}_j+\sigma,\dot{\delta}_j-\dot{\epsilon}_i-\sigma~|~ 1\leq i, j\leq n+1\}+\Bbb Z \delta .$$
We can extend the form $(-,-)$ on $\frak g=A(n,n)$ (see (\ref{norm})) to a form  on 
$\frak L$, which is denoted again by $(-,-)$, by the following rules
\begin{align}\label{form Ann}
&(c_i,d_j):=\delta_{i,j}\nonumber\\
&(c_i,c_j)=(d_i,d_j)=(c_i, {\frak g}\otimes \Bbb Ct^k)=(d_i,{\frak g}\otimes \Bbb Ct^k)=\{0\}\\
&(x\otimes t^k, y\otimes t^r):=\delta_{k+r,0}(x,y)\nonumber
\end{align}
where $1\leq i,j\leq 2$, $k,r\in \Bbb Z$ and $x,y\in \frak g$.

Table \ref{Table 3} demonstrates all kinds of roots of $ \frak L$.
\begin{table}[h!]
 \caption{Kinds of roots for type  $X= A(n,n)^{(1)}$}
\centering
 \begin{tabular}{||c| c| c| c||} 
 \hline
 $X^{(1)}$ & $R_{im}$ & $R_{re}$ & $R_{ns}$ \\ [0.5ex] 
 \hline\hline
$A(n-1,n-1)^{(1)}$ & $\Bbb Z \d$ & \makecell{$\{\dot{\epsilon}_i-\dot{\epsilon}_j,\dot{\delta}_i-\dot{\delta}_j\mid 1\leq i\neq j\leq n\}$\\$+\Bbb Z \delta$}& \makecell{$\{\pm (\dot\e_i-\dot\d_j+\sigma)\mid 1\leq i, j\leq n\}$\\$+\Bbb Z \d $}\\
 \hline
 \end{tabular}
 \label{Table 3}
\end{table}

According to Table \ref{Table 2} and Table \ref{Table 3}, if 
$S\subseteq R_{re}$ (or $\subseteq R_{ns}$), then 
$S^\times=S$, because
 $R_{im}=\Bbb Z\d$ for untwisted affine Lie superalgebras. Especially, $R_{re}^\times =R_{re}$ and 
 $R_{ns}^\times=R_{ns}$.

In the end, Table \ref{Table 4} provides the even and odd parts of the root systems of all untwisted affine Lie superalgebras.

\begin{table}[h!]
 \caption{The even and odd part of root systems for all types}
\centering
 \begin{tabular}{||c| c| c||} 
 \hline
 $X^{(1)}$ & $R_0$ & $R_{1}$   \\ [0.5ex] 
 \hline\hline
\makecell{$A(m-1,n-1)^{(1)}$,\\
   ($m\neq n$)} & \makecell{$\{\dot\e_i-\dot\e_r,\dot\d_j-\dot\d_s \mid 1\leq i,r\leq 
   m,$\\ $1\leq j,s\leq n \}+\Bbb Z \d$}& \makecell{$\pm\{\dot\e_i-\dot\d_j \mid 1\leq i\leq m,~1\leq j\leq n\}$\\$+\Bbb Z \d $}\\
 \hline
$A(n,n)^{(1)}$ &\makecell{$\{\dot{\epsilon}_i-\dot{\epsilon}_j,\dot{\delta}_i-\dot{\delta}_j\mid 1\leq i, j\leq n\}$\\$+\Bbb Z \delta$}& \makecell{$\{\pm (\dot\e_i-\dot\d_j+\sigma)\mid 1\leq i, j\leq n\}$\\$+\Bbb Z \d $}\\ 
 \hline
$B(m,n)^{(1)}$ & \makecell{$\pm \{\dot\epsilon_{i},  
\dot\epsilon_{i} \pm \dot\epsilon_{r}, \dot\delta_{j} \pm \dot\delta_{s} \mid 1 \leq i,r \leq 
m,$\\ $~ 1 \leq j,s \leq n, i \neq r\}+\Bbb Z\d$}&  \makecell{$\pm \{ \dot\delta_{j}, \dot\epsilon_{i} \pm\dot\delta_{j} \mid 1 \leq i \leq m, ~ 1 \leq j \leq n\}$\\
$+\Bbb Z\d$}\\
 \hline
 \makecell{$C(n)^{(1)}=$\\
 $D(1,n-1)^{(1)}$} &  \makecell{$\pm \{ 2\dot\epsilon_1 ,
 \dot\delta_{j} \pm \dot\delta_{s}\mid  1 \leq j,s \leq n-1\}$\\$+\Bbb Z\d$} & \makecell{$\pm \{ \dot\epsilon_1 \pm \dot\delta_{j}\mid  1 \leq j \leq n\}$\\$+\Bbb Z\d$} \\
 \hline
 \makecell{$D(m,n)^{(1)}$\\
 $(m\neq 1)$} &  \makecell{$\pm \{ \dot\epsilon_{i} \pm \dot\epsilon_{r},
 \dot\delta_{j} \pm \dot\delta_{s}\mid 1 \leq i,r \leq m,$\\$ 1 \leq j,s \leq n,  i \neq r
 \}+\Bbb Z\d$} & \makecell{$\pm \{ \dot\epsilon_{i} \pm \dot\delta_{j}\mid 1 \leq i \leq 
 m, 1 \leq j \leq n\}$\\$+\Bbb Z\d$} \\
 \hline
 $F(4)^{(1)}$ &  \makecell{$\pm \{ 0, \dot\epsilon, \dot\delta_{i} \pm 
 \dot\delta_{j}, \dot\delta_{i}\mid ~ 1 \leq i \neq j \leq 3  \}$\\$+\Bbb Z\d$} & 
 \makecell{$\pm \{\frac{1}{2} (\dot\epsilon \pm \dot\delta_{1} \pm \dot\delta_{2}\pm 
 \dot\delta_{3} )\}+\Bbb Z\d$} \\
 \hline
 $G(3)^{(1)}$ & \makecell{$\pm \{ 0,  2 \dot\nu , \dot\epsilon_{i} - \dot\epsilon_{j}, \dot\nu \pm (\dot\epsilon_{i} - \dot\epsilon_{j}) \mid \{i,$\\$j,t \} = \{ 1,2,3 \} \}+\Bbb Z\d$} & \makecell{$\pm \{ \dot\nu , 2 \dot\epsilon_{i} - \dot\epsilon_{j} - \dot\epsilon_{t}\mid \{ i,j,t \} = \{ 1,2,3 \} \}$\\$+\Bbb Z\d$} \\
 \hline
 $D(2,1;\lam)^{(1)}$ &  $\{ 0, \pm 2\dot\gamma_i \mid i\in \{ 1,2,3 \} 
 \}+\Bbb Z\d $& $\{\pm\dot\gamma_1\pm\dot\gamma_2\pm\dot\gamma_3\} +\Bbb Z\d$\\
 \hline
 \end{tabular}
 \label{Table 4}
\end{table}


\section{\bf Results }

Throughout this section, assume that $\frak L=\frak L_0\oplus \frak L_1$ represents an untwisted affine Lie superalgebra with the root system $R=R_0\cup R_1$ with respect to the standard Cartan subalgebra $\frak H$ of $\frak L_0$.

A vector (super)space $M$ together with a bilinear map 
$\cdot:\mathfrak {L}\times M\to M$ is called an ${\mathfrak {L}}$-{\it module} if
$$[x,y]\cdot u=x\cdot (y\cdot u)- (-1)^{\bar x\bar y}y\cdot (x\cdot u)$$
for all $x,y$ in 
$\mathfrak {g}$ and $u$ in $M$. 
An $\frak L$-module $M$ have a {\it weight space decomposition} with respect to $\frak H$ (or simply, a weight module)
if $M=\bigoplus_{\lambda \in \frak H ^*}M^\lambda$ where 
$\frak H^*$ is the dual space of $\frak H$, and 
$$M^\lambda :=\{x\in M~|~ hx=\lambda (h)x~ (\forall h\in \frak H)\}~~(\lambda \in \frak H^*).$$
The module $M$ is called a {\it finite weight module} if 
${\rm dim} M^\lambda <\infty$. The set 
$${\rm supp}(M):=\{\lambda \in {\frak H}^* ~|~ M^\lambda \neq {0}\}$$
is called {\it support} of module $M$. 
A subspace $N$ of $\frak L$-module $M$ is called {\it submodule} if 
$\frak L \cdot N\subseteq N$. A nonzero module $M$ is said to be {\it simple} if the only submodules are 
$\{0\}$ and $M$.

Suppose that $M$ is a simple $\frak L$-module and 
$\alpha \in R_{re}$. Then, consider the nonzero root vector corresponding to $\alpha$, say $x_\alpha$, and set
$$N:=\{v\in M~|~x_\alpha ^mv=0~ {\rm for~some~m\in \Bbb N}\}.$$ 
One can check that $N$ is
a submodule of $M$. Since $M$ is simple, either $N=0$ or $N=M$.
If $N=0$, then $x_\alpha$ acts injectively, otherwise $x_\alpha$ acts
locally nilpotently. This leads us to set some notations.
The set of all elements of $R_{re}$ whose corresponding nonzero root vectors act injectively (resp. locally nilpotently) on a weight $\frak L$-module $M$  denotes by $R^{in}_M$ (resp. $R^{ln}_M$).  If there is no ambiguity, we shall use the notation $R^{in}$ (resp. $R^{ln}$) instead of $R^{in}_M$ (resp. $R^{ln}_M$).

Based on Table \ref{Table 2} and  Table \ref{Table 4}, we have $R_0\cap R_{re}\neq \emptyset$ and some roots are located in both $R_{re}$ and $R_1$ (for some types of untwisted affine Lie superalgebra). Hence, $R_{re}\not \subseteq R_0$.
For $\a \in R_{re}$, by \cite[Example 3.4]{Yousofzadeh:Extended affine 2016} and \cite[Lemma 3.6]{Yousofzadeh:Extended affine 2016}, there are
$x\in \frak L^\alpha$, $y\in \frak L^{-\alpha}$ such that triple $(x,y,h:=[x,y])$ satisfies
$$[h,x]=2x~~{\rm and}~~[h,y]=-2y.$$ 
This triple is called an $\frak {sl}_2$-{\it super triple}.
The subsuperalgebra $\frak P$ generated by $\{x,y,h\}$ is isomorphic to $\frak{sl}(1,2)$ provided that 
$\a\in R_{re}\cap R_0$, and $\frak P$ is isomorphic to $\frak{osp}(1, 2)$ if $\a \in R_{re}\cap R_1$, see \cite[Section 2 and Section 3]{Yousofzadeh:Extended affine 2016}.

For real root $\alpha$ (and $h\in \frak H$), we set 
$$r_\alpha :{\frak H}^*\longrightarrow {\frak H}^*,~~\lambda \mapsto 
\lambda-\frac{2(\lambda , \alpha)}{(\alpha , \alpha)}\alpha~(=\lambda -\lambda(h)\alpha).$$

\begin{Lem}\label{3.1}
	Suppose that $M$ is a simple $\frak L$-module having a finite weight space decomposition with respect to $\frak H$ and corresponding representation $\pi$. Assume that $\alpha \in R_{re}\cap R_0$ and choose $x\in \frak L^\alpha$, $y\in \frak L^{-\alpha}$ such that $(x, y, h:=[x,y])$ is an $\frak {sl}_2$-triple. Assume that $x$ and $y$ 	
	act locally nilpotent on $M$. Then, for $\theta_\alpha:=e^{x}e^{-y}e^{x}$, we have
	$$\theta_\alpha(M^\lambda)=M^{r_\alpha (\lambda)}\quad \lambda \in {supp}(M),$$ 
	in particular, $\lambda \in {\rm supp}(M)$if and only if 
	$r_\alpha (\lambda)\in {\rm supp}(M)$.
\end{Lem}
\begin{proof}
	Repeat the proof of \cite[Lemma 3.4]{Yousofzadeh:Finite weight 2020}.  
\end{proof}

Let $M$ be an $\frak L$-module with a weight space decomposition with respect to $\frak H$. We define 
\begin{align*}
	{\frak B}_M&:=\left\{\alpha \in {\rm span}_{\Bbb Z}{R}~|~\{k\in \Bbb Z^{>0}~|~\lambda +k\alpha \in {\rm supp}(M)\}~{\rm is~finite}~\forall \lambda \in {\rm supp}(M)\right\},\\
	{\frak C}_M&:=\{\alpha \in {\rm span}_{\Bbb Z}{R}~|~
	\alpha +{\rm supp}(M)\subseteq {\rm supp}(M)\},\\
	\overline{{\frak B}}_M&:=\{\alpha \in {\rm span}_{\Bbb Z}{R}~|~t\alpha \in {\frak B}_M~ {\rm for ~some}~t\in \Bbb Z^{>0}\},\\
	\overline{{\frak C}}_M&:=\{\alpha \in {\rm span}_{\Bbb Z}{R}~|~t\alpha \in {\frak C}_M~ {\rm for ~some}~t\in \Bbb Z^{>0}\}.
\end{align*}

\begin{Lem}\label{3.2}
	The following statements hold.
	\begin{itemize}
		\item [(i)] $\alpha \in  {\frak B}_M$ if and only if for all positive integers $t$, $t\alpha \in {\frak B}_M$ if and only if there is a positive integers $t$ such that $t\alpha \in {\frak B}_M$. In particular, ${\frak B}_M=\overline{{\frak B}}_M$.
		\item [(ii)] If $\alpha_1 ,  \cdots, \alpha_n \in 
		{\frak C}_M$ (resp. $\overline{{\frak C}}_M$), then $\alpha_1 +  \cdots + \alpha_n \in 
		{\frak C}_M$ (resp. $\overline{{\frak C}}_M$).	
	\end{itemize}
\end{Lem}
\begin{proof}
	Just repeat the proof of \cite[Lemma  3.5]{Yousofzadeh:Finite weight 2020}.	
\end{proof}

\begin{Pro}\label{3.3}
	Suppose that $M$ is an $\frak L$-module having weight space decomposition with respect to $\frak H$ and suppose that $T$ satisfies one of the following:
 \begin{enumerate}
     \item[(i)] $T= R$;
     \item [(ii)] $T$ is a symmetric closed\footnote{That is $T=-T$ and $(T+T)\cap R\subseteq T$.}  subset of ${R}\setminus {R}_{ns}$ with $\Bbb Z\d\subseteq T$.
 \end{enumerate}
 Set $\mathcal{K}:= \bigoplus_{\a\in T} \frak L^{\a}$ (note that $\frak{H}=\frak{L}^0\subseteq \mathcal{K}$) and assume $W$ is a $\mathcal{K}$-submodule of $M$. Let $\emptyset \neq {\mathcal S} \subseteq {T}$ such that $\mathcal S$ does not contain imaginary roots, ${\mathcal S}\subseteq {\frak B}_W$ and 
	$-{\mathcal S}\subseteq {\frak C}_W$. Then the following hold.
	\begin{itemize}
		\item [(i)] If $\mathcal A$ is a nonempty set of ${\rm supp}(W)$ with 
		$$({\mathcal A}+{\mathcal S}) \cap {\rm supp}(M) \subseteq 
		{\mathcal A},$$
		then for each $\beta \in \mathcal S$, 
		$${\mathcal A}_\beta:=\{\lambda \in {\rm supp}(W)~|~ 
		\lambda +\beta \not \in {\rm supp}(M)\}$$
		is also nonempty with $({\mathcal A}_{\beta}+{\mathcal S}) \cap {\rm supp}(W) \subseteq 
		{\mathcal A}_\beta$.
		\item [(ii)] If $\mathcal S$ is finite and $\mathcal A$ is as in part (i), then there is $\lambda \in \mathcal A$ such that 
		$$(\lambda +{\rm span}_{\Bbb Z^{\geq 0}} {\mathcal S})\cap {\rm supp}(W)=\{\lambda\}.$$
	\end{itemize}
\end{Pro}
\begin{proof}
	Just repeat the proof of \cite[Proposition 3.6]{Yousofzadeh:Finite weight 2020}.	
\end{proof}

Recall that if $S$ is a subset of $R$ and  $f : {\rm span}_{\Bbb R} S \to \Bbb R$ is a linear functional, then the decomposition
$$ S = S^+\cup S^\circ\cup  S^-$$
where
$$S^\pm=S^\pm_f:=\{\a\in S\mid \pm f(\a)>0\} \quad {\rm and}\quad S^\circ=S^\circ_f:=\{\a\in S\mid f(\a)=0\}$$
is said to be a {\it triangular decomposition} for $S$.

Now, we have the following result, which has an important role in obtaining our main result, Theorem \ref{3.9}.

\begin{Pro}\label{3.4}
Suppose that $M$ is an $\frak L$-module having weight space decomposition with respect to $\frak H$ and suppose that $T$ satisfies one of the following:
 \begin{enumerate}
     \item[(i)] $T= R$;
     \item [(ii)] $T$ is a symmetric closed subset of ${R}\setminus {R}_{ns}$ with $\Bbb Z\d\subseteq T$.
 \end{enumerate}
 Set $\mathcal{K}:= \bigoplus_{\a\in T} \frak L^{\a}$ (note that $\frak{H}=\frak{L}^0\subseteq \mathcal{K}$) and assume $W$ is a $\mathcal{K}$-submodule of $M$.
	Suppose that 
	$T=T^+\cup T^0\cup T^-$ is a triangular decomposition for $T$ with corresponding linear functional $\zeta :{\rm span}_\Bbb R T\longrightarrow \Bbb R$ where $\zeta (\delta)\neq 0$. Set
	$$T_{re}:=T\cap R_{re},~~T_{re} ^\pm:=T^\pm \cap R_{re},~{\rm and}~\dot{T}=\{\dot \a \in \dot R~|~ (\dot \a +\Bbb Z \d)\cap T\neq \emptyset\}.
 $$
	Assume that  $T^+_{re}\subseteq {\frak B}_W, ~ T^-_{re}\subseteq {\frak C}_W$ and $\dot T_{re}\neq \emptyset$. If $p\in \Bbb Z^{>0}$ and $\lambda \in 
	{\rm supp}(W)$ with $(\lambda +\Bbb Z^{>0}p\delta)\cap 
	{\rm supp}(W)=\emptyset$, then there is $\mu\in {\rm supp}(W)$ such that 
	$$\left(\mu +(T^+_{re}\cup \Bbb Z^{>0}\delta)\right)\cap {\rm supp}(W)=\emptyset.$$
\end{Pro}

\begin{proof}
Assume that $\zeta (\delta)> 0$. 
Fix $p$ and $\lambda$ as mentioned.
	Let  $\dot{\beta}\in \dot R_{re}$. Then 
	for large enough $k>0$, we have
	$$0<\zeta (\dot{\beta})+k\zeta (\delta)=\zeta (\dot{\beta}+k\delta).$$
	Throughout the proof, $t_{\dot{\beta}}$ shall denote the smallest positive integer with the property $\zeta (\dot{\beta}+t_{\dot{\beta}}\delta)>0$. 
 Define
	\begin{align*}
		&{\mathcal P}:=\{\dot{\alpha}+(t_{\dot{\alpha}}+s)\delta~|~
		\dot{\alpha}\in \dot{T}_{re},~ 0\leq s\leq p\}\cap T,\\
		&{\mathcal S}:=\{\dot{\alpha}+t_{\dot{\alpha}}\delta~|~\dot{\alpha}\in \dot{T}_{re}\},\\
		&{\mathcal A}:=\left\{\mu\in {\rm supp}(W)~|~
		\{\alpha \in T_{re}^+~|~
		\mu+\alpha\in{\rm supp}(W)\}\subseteq {\mathcal P}\right\}.
	\end{align*}
	Clearly, ${\mathcal S}\subseteq {\mathcal P}\subseteq T_{re}^+$. Furthermore, if $\mu_0\in\mathcal A$, 
	then one can actually infer that the set
	$$\{\alpha \in T^+_{re}~|~ \mu_0 +\alpha \in 
	{\rm supp}(W)\}\quad\quad\quad(*)$$
	is finite. We now break the proof into the following claims.
	
	{\bf Claim 1.} $\mathcal A \neq \emptyset$. To see this, we show that  $\lambda \in \mathcal A$. Let $\alpha \in T_{re}^+$ with 
	$\lambda +\alpha \in {\rm supp}(W)$. We must prove that 
	$\alpha \in \mathcal P$. Since 
	$\alpha \in T_{re}^+$, we can assume that 
	$\alpha =\dot{\alpha}+m\delta$ for some $\dot{\alpha}\in \dot{T}_{re}$ and $m\in \Bbb Z$ with $m\geq t_{\dot{\alpha}}$.
	We can write $m-t_{\dot{\alpha}}=pq+r$ where $q\in \Bbb Z^{\geq 0}$ and 
	$0\leq r\leq p-1$. Assume that $q\neq 0$.
	Then, we have
	\begin{align*}
		&\lambda +\dot{\alpha}+(t_{\dot{\alpha}}+r)\delta +qp\delta=
		\lambda +\dot{\alpha}+t_{\dot{\alpha}}\delta +(m-t_{\dot{\alpha}})\delta=\\
		&\lambda +\dot{\alpha}+(t_{\dot{\alpha}} +m-t_{\dot{\alpha}})\delta=
		\lambda +(\dot{\alpha}+m\delta)=\lambda +\alpha \in {\rm supp}(W).
	\end{align*}
	Clearly, $-\left(\dot{\alpha}+(t_{\dot{\alpha}}+r)\delta\right)\in T^-_{re}$. 
	By the assumption 
	$T^-_{re}\subseteq \frak C_M$, hence the element
	\begin{align*}
		&-\left(\dot{\alpha}+(t_{\dot{\alpha}}+r\right)\delta)+\lambda +\alpha=
		-(\dot{\alpha}+(t_{\dot{\alpha}}+r)\delta)+\lambda +\dot{\alpha}+(t_{\dot{\alpha}}+r)\delta +qp\delta=\lambda +qp\delta  
	\end{align*}
	is contained in ${\rm supp}(W)$. On the other hand, $pq>0$. It follows that $\lambda +qp\delta \in (\lambda +\Bbb Z^{>0}p\delta)\cap 
	{\rm supp}(W)=\emptyset$, a contradiction. Thus $q=0$. This implies that 
	$\alpha =\dot{\alpha}+(t_{\dot{\alpha}}+r)\delta \in \mathcal P$, and so $\lambda \in \mathcal A$. 
	
	{\bf Claim 2.} For any $\mu \in \mathcal A$, the set 
	$$Q:=\{m\delta ~|~ m\in \Bbb Z^{>0},~~\mu+m\delta \in {\rm supp}(W)\}$$
	is finite. Suppose on the contrary that for some $\mu_0\in \mathcal A$,
	there are infinitely many $m\delta \in \Bbb Z^{>0}\delta$ with 
	$$\mu_0+m\delta \in {\rm supp}(W).$$
Suppose that $\dot\a_*\in \dot T_{re}$. Then, $\dot\a_*+t_{\dot\a_*}\d\in T^+_{re}$.
	Thus, by choosing infinitely many
$n\d\in Q$ with
$n\geq  t_{\dot\a_*}$, we have $\mu_0+n\d\in {\rm supp}(W)$. As
$-(\dot\a_*+t_{\dot\a_*}\d)\in T^-_{re} \subseteq {\frak C}_M$, we get 
$$-(\dot\a_*+t_{\dot\a_*}\d)+\mu+n\d=\mu+(-\dot\a_*+(n-t_{\dot\a_*}))\d\in {\rm supp}(W)$$
for infinitely many $n$. So, 
$-\dot\a_*+(n-t_{\dot\a_*})\d\in T^+_{re}$
 for sufficiently large numbers $n$.
	This contradicts the fact $(*)$.
	
	{\bf Claim 3.} There is $\eta\in {\rm supp}(W)$ such that $\eta+m\delta \not \in {\rm supp}(W)$ for all $m\in \Bbb Z^{>0}$. 
	To see this, by Claim 1, pick $\mu_0 \in {\mathcal A}\subseteq {\rm supp}(W)$. By 
	Claim 2, we can choose the greatest $N\in \Bbb Z^{>0}$ such that
	$\mu_0+N\delta \in {\rm supp}(W)$. Set $\eta=\mu_0+N\delta$. It is clear that 
	$\eta+m\delta \not\in {\rm supp}(W)$ for all $m\in\Bbb Z^{>0}$.

	{\bf Claim 4.} There is $\eta\in X$ such that 
	$(\eta+{\rm span}_{\Bbb Z^{\geq 0}}{\mathcal S})\cap  {\rm supp}(W)=\{\eta\}$ where
	$$X:=\{\mu\in {\rm supp}(W)~|~\forall m\delta \in Z^{>0}\delta,~~
	\mu+m\delta\not \in  {\rm supp}(W)\}.$$
First, we note that, by Claim 3, $X$ is a nonempty set.  
	We apply Proposition  \ref{3.3}(ii) to prove this claim. 
	Assuming this proposition, we must have 
	$$(\eta+{\rm span}_{\Bbb Z^{\geq 0}}{\mathcal S})\cap  {\rm supp}(W)=\{\eta\},$$
	as claimed. Now, let us show how to apply this proposition. 
		It is clear that 
	$\mathcal S$ is a nonempty subset of $T$. 
	Note that $R_{im}=\Bbb Z\delta$ and $R_{im}^\pm:= R_{im}\cap R^\pm
	={\Bbb Z}^{{\gtrless} 0}\delta$, because $\zeta (\delta)>0$,  and $\zeta (k\delta)=k\zeta (\delta)$ for all
	$k\in\Bbb Z$. Thus, $\mathcal S$ does not contain imaginary roots.
	According to the assumption and material mentioned above, we have ${\mathcal S}\subseteq {\mathcal P}\subseteq T_{re}^+ \subseteq \frak B_W$ and $T_{re}^-\subseteq \frak C_M$, and hence $-{\mathcal S}\subseteq \frak C_M$.  It only remains to show that 
	$$(X+{\mathcal S})\cap {\rm supp}(W)\subseteq X$$
	is true so that we must meet all the conditions of Proposition  \ref{3.3}(ii). Assuming the contrary,  there are  $\eta \in X$ and  $\beta \in \mathcal S$ such that $\eta +\beta \in {\rm supp}(W)$ and $\eta +\beta \not \in X$. Hence, there is  $ m\delta \in Z^{>0}\delta$ with $\eta +\beta +m\delta \in {\rm supp}(W)$. Since 
	$-\beta \in \frak C_W$, we can deduce that  
	$$\eta +m\delta=\eta +\beta +m\delta-\beta \in {\rm supp}(W),$$
	which contradicts the fact that $\eta \in X$.
	
	{\bf Claim 5.} There is  $\mu \in {\rm supp}(W)$ such that 
	$$\left(\mu +(T^+_{re}\cup \Bbb Z^{>0}\delta)\right)\cap {\rm supp}(W)=\emptyset.$$
	To prove this, we can deduce from Claim 4 that there is at least one $\mu$ with the following property
	$$\left(\mu+(\Bbb Z^{>0}\delta \cup{\rm span}_{\Bbb Z^{\geq 0}}{\mathcal S})\right)\cap  {\rm supp}(W)=\{\mu\}\quad\quad\quad (**)$$
	Now, suppose on the contrary that there is  
	$\alpha \in T^+_{re}\cup \Bbb Z^{>0}\delta$ with $\mu +\alpha \in {\rm supp}(W)$.
	If $\alpha \in \Bbb Z^{>0}\delta$, then $\mu +\alpha =\mu$ by $(**)$. This implies that $\alpha =0$, a contradiction.
	Hence $\alpha \in T^+_{re}$, and so 
	$\alpha =\dot{\alpha}+m\delta$ for some $\dot{\alpha}\in \dot{T}^\times _{re}$ and $m\geq t_{\dot{\alpha}}$.  Since $-(\dot{\alpha}+t_{\dot{\alpha}}\delta)
	\in T^-_{re}\subseteq \frak C_W$, we have
	$$
	\mu+(m-t_{\dot{\alpha}})\delta=\mu+\alpha-(\dot{\alpha}+t_{\dot{\alpha}}\delta)\in {\rm supp}(W).$$
	Thus $\mu+(m-t_{\dot{\alpha}})\delta =\mu$ by $(**)$. This implies that $m=t_{\dot{\alpha}}$. Hence, $\alpha \in \mathcal S$. This  again leads us to
	$\mu+\alpha=\mu$ by $(**)$, a contradiction. 
	Therefore,  there is no
	$\alpha \in T^+_{re}\cup \Bbb Z^{>0}\delta$ with $\mu +\alpha \in {\rm supp}(W)$, and the proof is complete.\\
 If $\zeta (\d)<0$, then we use $-\d$ instead of $\d$ and repeat the above process.
\end{proof}

\begin{Pro}\label{3.5}
	Assume that $M$ is an $\frak L$-module, $\zeta$ is a linear functional on ${\rm span}_{\Bbb R}R$ with corresponding triangular decomposition $R=R^+\cup R^0\cup R^-$, and  $\zeta (\delta)>0$ {\rm (}$\zeta (\sigma)=0$, for the case $\frak L=A(n,n)^{(1)}${\rm )}. Set
$${\mathcal A}:=\{0\neq v\in M~|~ {\frak L}^\alpha v=\{0\}, \forall \alpha \in R^+\cap (R_{re}\cup \Bbb Z \d)\}$$
and 
$$B:=\{v\in {\mathcal A}~|~\forall \dot{\alpha}\in \dot{R}_{ns}~~\exists N\in \Bbb Z^{\geq 0}~;~{\frak L}^{\dot{\alpha}+r_{\dot\a}\sigma+n\delta}v=\{0\}~~~~(\forall n\geq N)\}.
$$
in which for $\dot{\alpha}\in \dot{R}_{ns}$, 
$$
r_{\dot\a}:=\begin{cases}
1&\quad \dot\a=\dot\e_i-\dot\d_j ~~{\rm and}~~\frak L=A(n,n)^{(1)},\\
-1&\quad \dot\a=\dot\d_j-\dot\e_i ~~{\rm and}~~\frak L=A(n,n)^{(1)},\\
0&\quad {\rm otherwise}.
\end{cases}
$$
If $B$ is nonempty, then 
	$$M^{\frak L^+}=\{v\in M~|~\frak L^\alpha v=\{0\}, ~~\forall \alpha \in R^+\}\neq \{0\}.$$	
\end{Pro}
\begin{proof}
	Note that
	$${\mathcal A}=\{0\neq v\in M~|~ \frak L^{n\delta}v=
	{\frak L}^\alpha v=0, \forall \alpha \in R^+\cap R_{re},~n\in \Bbb Z^{>0}\}.$$	
	Let  $0\neq \beta\in R$. If $\b \in R_{ns}$, then 
	$\beta=\dot{\beta}+r_{\dot\b}\sigma+n\delta$ 
	in which $\dot{\beta}\in \dot{R}_{ns}$ and $n\in \Bbb Z$. 
	Since $\zeta (\delta)>0$, we can find   $k\in \Bbb Z$ such that
	$$0<\zeta (\beta)+k\zeta (\delta)=\zeta (\beta+k\delta).$$
 The same argument holds for $\b \not\in R_{ns}$. Throughout the proof, $m_{\dot{\beta}}$ shall denote the smallest integer with the property 
	$$\zeta (\dot{\beta}+m_{\dot{\beta}}\delta)>0,$$
	and we define $\b_{\dot{\a}}:=\dot{\a}+r_{\dot\a}\sigma +m_{\dot{\a}}\delta$,
	and  we set $\Phi$ as the set of all such $\b_{\dot{\a}}$s. 
	We now break the proof into the following claims.
	
	{\bf Claim 1.} $B$ is equal to 
	$$B':=\{v\in{\mathcal A}~|~\exists N\in \Bbb Z^{\geq 0};~ {\frak L}^{\alpha+n\delta}v=0~~
	(\alpha\in \Phi\cap R_{ns},~n\geq N)\}.$$
	To see this, let $v\in B$. Thus, for each $\dot{\alpha}\in \dot{R}_{ns}$, there is 
	 $N_{\dot{\alpha}}\in \Bbb Z^{\geq 0}$ with ${\frak L}^{\dot \a+r_{\dot\a}\sigma+n\delta}v=0$ for all $n\geq N_{\dot{\alpha}}$.
	Set $N:=\max\{N_{\dot{\alpha}}-m_{\dot{\alpha}}~|~\dot{\alpha}\in \dot{R}_{ns}\}$. Then, 
	${\frak L}^{\beta_{\dot{\alpha}}+n\delta}v=0$ for all $n\geq N$ and 
	$\dot{\alpha}\in \dot{R}_{ns}$. This means $B\subseteq B'$.
	Conversely, suppose $v\in B'$ and pick $N\in \Bbb Z ^{\geq 0}$ with
	${\frak L}^{\beta_{\dot{\alpha}}+n\delta}v=0$ for all $n\geq N$ and 
	$\dot{\alpha}\in \dot{R}_{ns}$.
	Hence,  for each $\dot{\alpha}\in \dot{R}_{ns}$ and 
	$n\geq N+m_{\dot{\alpha}}$,
	we have ${\frak L}^{\beta_{\dot{\alpha}}+n\delta}v=0$, that is $v\in B$.
	
	Claim 1 leads us to define the 
	$$n_v:=\min\{N\in \Bbb Z^{\geq 0}~|~{\frak L}^{{\alpha}+n\delta}v=0~~(\alpha\in \Phi\cap R_{ns},~n\geq N)\}$$
	and
	$$C_v:=\{\alpha+t\delta~|~\alpha\in \Phi\cap R_{ns},~0\leq t<n_v\} \subseteq R_{ns}.$$
	
	{\bf Claim 2.} Assume $v\in B$, $N\in \Bbb Z^{\geq 0}$ and $\alpha \in C_v$ with the following properties
	\begin{itemize}
		\item [(1)] ${\frak L}^{{\alpha}+N\delta}v\neq 0$,
		\item[(2)] If $\alpha'\in C_v$ and ${\frak L}^{{\alpha'}+N\delta}v\neq 0$, then $\zeta (\alpha')\leq \zeta (\alpha)$,
		\item [(3)] For all $m\in \Bbb Z^{\geq 1}$ and $\alpha'\in C_v$,
		${\frak L}^{{\alpha'}+N\delta+m\delta}v=0$.
	\end{itemize}
	Then for $0\neq w\in {\frak L}^{{\alpha}+N\delta}v$, $w\in B$.\\	
	To prove this, we first show that ${\frak L}^{m\delta}w=0$ for all $m\in \Bbb Z^{\geq 1}$. Now, we have 
	$${\frak L}^{m\delta}w\subseteq {\frak L}^{m\delta}{\frak L}^{\alpha+N\delta}v\subseteq \underbrace{{\frak L}^{\alpha+(N+m)\delta}v}_{=0,~{\rm by~(3)}}+
	{\frak L}^{\alpha +N\delta}\underbrace{{\frak L}^{m\delta}v}_{=0 ~(v\in \mathcal A)}=0.
	$$
	Now, let $\beta \in R_{re}$ with $\zeta (\beta)>0$. 
	Since $v\in \mathcal A$, we have ${\frak L}^{\beta}v=0$. Thus
	$${\frak L}^{\beta}w\subseteq {\frak L}^{\beta}{\frak L}^{\alpha +N\delta}v\subseteq {\frak L}^{\alpha+\beta+N\delta}v+{\frak L}^{\alpha+N\delta}{\frak L}^{\beta}v={\frak L}^{\alpha+\beta+N\delta}v.$$
	If $\alpha+\beta+N\delta$ is not a root, then ${\frak L}^{\alpha+\beta+N\delta}v=0$, and so ${\frak L}^{\beta}w=0$. Assume that 
	$\alpha+\beta+N\delta \in R_{re}$. 
	Since $\alpha \in C_v$, one can show that $\zeta (\alpha)>0$. On the other hand, $\zeta(\beta)>0$ by the assumption. Thus 
	$$\zeta (\alpha+\beta+N\delta)=\zeta(\alpha)+\zeta(\beta)+N\zeta (\delta) >0.$$
	It follows that $\alpha+\beta+N\delta\in R^+\cap R_{re}=R_{re}^+$, and hence ${\frak L}^{\alpha+\beta+N\delta}v=0$ (because $v\in\mathcal A$). This implies that ${\frak L}^{\beta}w=0$.\\	
	Now, assume that $\alpha+\beta+N\delta\in R_{ns}$. This forces that $\alpha+\beta \in R_{ns}$. Thus $\alpha +\beta =\dot{\gamma}+ r_{\dot\gamma}\sigma +k\delta$ where $\dot{\gamma}\in \dot{R}_{ns}$ and $k\in \Bbb Z$. Since $\zeta (\alpha+\beta)>0$,   we can write 
	$\alpha +\beta=\dot{\gamma}+r_{\dot\gamma}\sigma+(m_{\dot{\gamma}}+m)\delta$ where $m\in \Bbb Z^{\geq 0}$. If 
	$m=0$, then $\alpha +\beta \in C_v$. However, 
	$\zeta (\alpha +\beta)=\zeta (\alpha)+\zeta (\beta)> \zeta (\alpha)$. 
	By (2), we get ${\frak L}^{\alpha+\beta+N\delta}v= 0$, and so ${\frak L}^{\beta}w=0$. If $m>0$, then by (3), ${\frak L}^{\alpha+\beta+N\delta}v= 0$, and so ${\frak L}^{\beta}w=0$.
	Thus, ${\frak L}^{\beta}w=0$ for all $0\neq \beta \in R^+\cap R_{re}$.
	
	Now, we show that $w\in B=B'$ (see Claim 1). Thus by the above argument, it is enough to prove that there is a positive integer $p$ such that for all $\eta \in \Phi \cap R_{ns}$ and $n\geq p$, ${\frak L}^{\eta+n\delta}w=0$. Since $v\in B=B'$, we can choose $p\in \Bbb Z^{>0}$ such that ${\frak L}^{\eta+n\delta}v=0$ for all $\eta \in \Phi \cap R_{ns}$ and $n\geq p$. Hence for all $\eta \in \Phi \cap R_{ns}$ and $n\geq p$,
	$${\frak L}^{\eta+n\delta}w\subseteq {\frak L}^{\eta+n\delta}{\frak L}^{\alpha+N\delta}v\subseteq{\frak L}^{\eta+\alpha+n\delta+N\delta}v+{\frak L}^{\alpha+N\delta}\underbrace{{\frak L}^{\eta+n\delta}}_{=0}v={\frak L}^{\eta+\alpha+n\delta+N\delta}v.$$
	Since $C_v\subseteq R_{ns}$, it follows that $\alpha$ and $\eta$ belong to $R_{ns}$. 
If $\eta+\alpha+n\delta+N\delta$ is not a root, then ${\frak L}^{\eta+\alpha+n\delta+N\delta}v=0$, and so 
	${\frak L}^{\eta+n\delta}w=0$. Hence, assume that $\eta+\alpha+n\delta+N\delta\in R$. 
 We know that the sum of two nonsingular roots is either a real root or an imaginary root. 
 If $\eta+\alpha+n\delta+N\delta\in R_{im}$, then we must have 
 $\eta+\a=0$.
Since $\a,\eta\in \Phi$, it follows that $\zeta(\a+\eta)>0$, which contradicts the condition $\a+\eta=0$.
Thus,  $\eta+\alpha+n\delta+N\delta\in R_{re}$. On the other hand,  
	$\zeta (\eta+\alpha+n\delta+N\delta)>0$, that is $\eta+\alpha+n\delta+N\delta\in R^+_{re}$. This implies that  ${\frak L}^{\eta+\alpha+n\delta+N\delta}v=0$, because $v\in \mathcal A$. Therefore, 
	${\frak L}^{\eta+n\delta}w=0$.
	
	Now, just repeat Claim 3, Claim 4, Claim 5 and Claim 6 of 
	\cite[Proposition 3.8]{Yousofzadeh:Finite weight 2020}.
\end{proof}

We recall that a weight module $H$ over untwisted affine Lie superalgebra  $\frak L$ has {\it shadow}, if 
$$R_{re} =R^{in}\cup R^{ln},~~
R^{ln}={\frak B}_H\cap R_{re},~~{\rm and}~~
R^{in}={\frak C}_H\cap R_{re}.$$

The following provides an example of modules having shadows.

\begin{Pro}\label{3.6}
	The following statements hold.
	\begin{itemize}
		\item [(i)] If $\frak L$-module $M$ is simple, then 
		$R_{re}=R^{in}\cup R^{ln}$.
		\item [(ii)] If finite weight $\frak L$-module $M$ satisfies the condition $R_{re}=R^{in}\cup R^{ln}$, then $M$ has shadow. 
	\end{itemize}
In particular, all simple finite weight $\frak L$-modules have shadow.
\end{Pro}
\begin{proof}
	Repeat the proof of \cite[Proposition 4.4]{Yousofzadeh:Finite weight 2020}.
\end{proof}

\begin{Lem}\label{3.7}
	Suppose that $M$ has shadow and $\alpha \in R_{re}$. Then the following statements hold. 
	\begin{itemize}
		\item [(i)] $\alpha \in {\frak C}_M$ if and only if $t\alpha \in {\frak C}_M$ for some $t\in \Bbb Z^{>0}$.
		\item[(ii)] If either $\pm\alpha\in R^{ln}$ or 
		$\pm\alpha  \in R^{in}$, then for $\gamma \in R_{re}$, $\gamma \in R^{in}$ if and only if $r_\alpha (\gamma)\in R^{in}$.
	\end{itemize}	
\end{Lem}

\begin{proof}
	(i). If  $\alpha \in {\frak C}_M$, then by Lemma \ref{3.2}(ii), 
	$$\underbrace{\alpha+\cdots +\alpha}_{t-{\rm times}}=t\alpha\in \frak C_M$$
	for any $t\in \Bbb Z^{>0}$.	 To prove the converse, let $t\alpha \in {\frak C}_M$ for some $t\in \Bbb Z^{>0}$. Since $\alpha \in R_{re}=R^{in}\cup R^{ln}$, we have $\alpha \in R^{in}$ or 
	$\alpha \in  R^{ln}$. If $\alpha \in  R^{ln}$, then $\alpha \in \frak B_M$. This implies that there are finitely many positive integers $k_1, \cdots, k_m$ with $\eta +k_i\alpha \in {\rm supp}(M)$, $1\leq i\leq m$, for some $\eta \in {\rm supp}(M)$. As $t\alpha \in {\frak C}_M$, we have 
	$$t\alpha+\eta +k_1\alpha=\eta +(t+k_1)\alpha \in {\rm supp}(M).$$
	Again, as $t\alpha \in {\frak C}_M$, we have 
	$$t\alpha+\eta +(t+k_1)\alpha=\eta +(2t+k_1)\alpha \in {\rm supp}(M).$$
	This process can continue unabated. Thus for large enough $i\in \Bbb N$, $\eta +(it+k_1)\alpha$ is distinct from
	$\eta +k_1\alpha$, $\eta +k_2\alpha$, $\cdots$, and $\eta +k_m\alpha$, which is contradiction. Thus $\alpha \in R^{in}\subseteq \frak C_M$.\\	
	(ii). We first note that, according to Table \ref{Table 2} and  Table \ref{Table 4}, in types $B(m,n)^{(1)}$ and $G(3)^{(1)}$, some roots belong to both $R_{re}$ and $R_1$.
 Let $\a \in R_{re}\cap R_1$. It is clear that $2\a\in R_{re}\cap R_0$. On the other hand,  there are $x\in \frak L^\a$ and $y\in \frak L^{-\a}$ such that
${\rm span}_{\Bbb C}\{x, y,h := [x, y], [x, x], [y, y]\}$
is a Lie superalgebra isomorphic to $\frak {osp}(1, 2)$ with $\a(h) =2$ (we refer to the material before Lemma \ref{3.1}).
 This leads us to the $\frak{sl}_2$-triple $(\dfrac{1}{4}[x,x],\dfrac{-1}{4}[y,y], \dfrac{1}{2}h)$ for $2\a$, and consequently $r_\a=r_{2a}$. Thus, by (i), it is enough to prove the statement (2) for $\a\in R_{re}\cap R_0$.
 As well, we note that if $\a, \gamma\in R_{re}$ then $r_\a(\gamma)\in R_{re}$.\\
	Now, assume that $\pm\alpha\in R^{ln}$.
Then,
	\begin{align*}
		\gamma \in R^{in}&\Longrightarrow \gamma \in \frak C_M \Longrightarrow 
		\forall n\in \Bbb Z^{\geq 0}~\forall \lambda\in {\rm supp}
		(M),~\lambda+n\gamma \in{\rm supp}(M)\\
  &\xRightarrow{{\rm Lemma}~\ref{3.1}}
	 \forall n\in \Bbb Z^{\geq 0}~\forall \lambda\in {\rm supp}
		(M),~r_\alpha(\lambda)+nr_\alpha(\gamma) \in{\rm supp}(M)\\
		&\Longrightarrow r_\alpha(\gamma) \in \frak C_M \Longrightarrow r_\alpha(\gamma) \in R^{in}.
	\end{align*}
	Because the opposite direction of the arrows is also correct, the proof of the case $\pm\alpha\in R^{ln}$ is complete.
	
	Now, assume that $\pm\alpha\in R^{in}$. If $\gamma\in R^{in}$, then
	by Lemma \ref{3.2}(ii), $r_\alpha(\gamma)=\gamma +m\alpha \in \frak C_M$. Hence, $r_\alpha(\gamma)\in R^{in}$. Conversely, assume that
	$\gamma +m\alpha=r_\alpha(\gamma)\in R^{in}$.
  Since $\pm\alpha\in R^{in}$, we have $\pm m\a\in \frak C_M$ by Lemma \ref{3.2}(ii).
 Using Lemma \ref{3.2}(ii) again, $\gamma=\underbrace{-m\a}_{\in \frak C_m}+\underbrace{\gamma +m\alpha}_{\in \frak C_m}\in \frak C_M$.
 So, $\gamma \in R^{in}$.
\end{proof}

\begin{The}\label{3.8}
	Suppose that $M$ is an $\frak L$-module having shadow. Then 
	\begin{itemize}
		\item [(i)] $(R^{ln}+R^{ln})\cap R_{re} \subseteq R^{ln}$.
		\item [(ii)] $(R^{ln}+2R^{ln})\cap R_{re} \subseteq R^{ln}$.
	\end{itemize}
\end{The}

\begin{proof}
	Repeat the proof of \cite[Theorem 4.7]{Yousofzadeh:Finite weight 2020} and use the above results (presented for the untwisted cases) in the body of that proof.
%
\end{proof}

Now we state the main theorem of this note.

\begin{The}\label{3.9}
	Suppose that $M$ is an $\frak L$-module having shadow. Then, 
	for each $\beta \in R_{re}$, one of the following will happen.
	\begin{itemize}
		\item [(i)] $\beta +\Bbb Z\delta \subseteq R^{ln}$.
		\item [(ii)] $\beta +\Bbb Z\delta \subseteq R^{in}$.
		\item [(iii)] there exist $m\in \Bbb Z$ and $t\in \{-1,0,1\}$ such that for $\gamma:=\beta +m\delta$, 
		\begin{align*}
			&\gamma +\Bbb Z^{\geq 1}\delta\subseteq R^{in},~~\gamma +\Bbb Z^{\leq 0}\delta\subseteq R^{ln}\\
			&-\gamma +\Bbb Z^{\geq t}\delta\subseteq R^{in},~~-\gamma +\Bbb Z^{\leq t-1}\delta\subseteq R^{ln}
		\end{align*}
		\item [(iv)] there exist $m\in \Bbb Z$ and $t\in \{-1,0,1\}$ such that for $\eta:=\beta +m\delta$, 
		\begin{align*}
			&\eta +\Bbb Z^{\leq -1}\delta\subseteq R^{in},~~	\eta +\Bbb Z^{\geq 0}\delta\subseteq R^{ln}\\
			&-\eta +\Bbb Z^{\leq -t}\delta\subseteq R^{in},~~-\eta +\Bbb Z^{\geq 1-t}\delta\subseteq R^{ln}.
		\end{align*}
	\end{itemize}
\end{The}

\begin{proof}
	We can write $\beta=\dot{\beta}+n\delta$ in which $\dot{\beta}\in \dot{R}_{re}$ and $n\in \Bbb Z$. Clearly, 
	$$(\beta+\Bbb Z\delta)\cap R=\beta+\Bbb Z\delta\subseteq R_{re}=R^{ln}\cup R^{in}.$$ 
	Thus, if (i) and (ii) do not occur, then
	$$\gamma:=\beta +k\delta \in R^{ln},~~\gamma +\delta=\beta+(k+1)\delta\in R^{in}\quad\quad (*)$$
	or
	$$\gamma:=\beta +k\delta \in R^{in},~~\gamma +\delta=\beta+(k+1)\delta\in R^{ln} \quad\quad (**)$$
	for some $k\in \Bbb Z$.
	We claim that if $(*)$  holds, then (iii)  happens.
Before starting the main body of the proof of the claim, we show that the situation $$-\gamma \in R^{ln},{\rm and}~~-\gamma -\delta \in R^{in}$$
 cannot happen. Suppose on the contrary that
	$-\gamma \in R^{ln}$ and $-\gamma -\delta \in R^{in}$. Then, according to $(*)$, we have $\pm\gamma\in R^{ln}$ or $\pm(\gamma+\delta)\in R^{in}$. 
	Thus by Lemma \ref{3.7}(ii)
	\begin{align*}
		&r_{\gamma}(\gamma+\delta)=\gamma+\delta-\frac{2(\gamma, \gamma+\delta)}{(\gamma,\gamma)}\gamma=-\gamma+\delta\in R^{in},\\
		&r_{\gamma}(-\gamma-\delta)=-\gamma-\delta-\frac{2(\gamma, -\gamma-\delta)}{(\gamma,\gamma)}\gamma=\gamma-\delta\in R^{in}.
	\end{align*}
Let $\mu \in {\rm supp}(M)$. As $-\gamma+\d\in R^{in}$, it follows that $\mu+(-\gamma+\d)\in {\rm supp}(M)$. 
On the other hand, since 
	$\gamma -\delta\in R^{in}\subseteq \frak C_M$, we have $\mu\in {\rm supp}(M)$. So, by similar arguments,   we can easily show that
	$$\mu+(-\gamma+\delta)\in {\rm supp}(M)\Leftrightarrow
	\mu\in {\rm supp}(M)\Leftrightarrow
	\mu+(\gamma+\delta)\in {\rm supp}(M).$$
	Now, let  $\eta \in {\rm supp}(M)$. Since $\gamma \in R^{ln}\subseteq \frak B_M$, we can choose $k\in \Bbb Z^{\geq 1}$ such that 
	\begin{align*}
		&\eta+2k\gamma \not \in {\rm supp}(M) \Longrightarrow\\
		&\eta+2k\gamma+2k(-\gamma+\delta)\not \in {\rm supp}(M) \Longrightarrow\\
		&\eta+2k\delta\not \in {\rm supp}(M) \Longrightarrow\\
		&\eta+k(\gamma+\delta)+k(-\gamma+\delta)\not \in {\rm supp}(M) \Longrightarrow\\
		&\eta \not \in {\rm supp}(M),
	\end{align*}
	which is a contradiction.  We now proceed with the remaining cases for $-\gamma$  and $-\gamma -\delta$. 
	
	{\bf Case 1.} $-\gamma \in R^{ln}$ and $-\gamma -\delta \in R^{ln}$.\\
Assume that there is $p\in\Bbb Z$ such that $\gamma+p\d\in R^{in}$.  AS 
$\pm\gamma \in R^{ln}$, it follows from Lemma \ref{3.7}(ii) that 
$$r_\gamma(\gamma+p\d)=\gamma+p\d-\dfrac{2(\gamma,\gamma)}{(\gamma,\gamma)}\gamma
=-\gamma+p\d\in R^{in}.$$
 Again by Lemma \ref{3.7}(ii),
 $$r_\gamma(-\gamma+p\d)=-\gamma+p\d-\dfrac{2(-\gamma,\gamma)}{(\gamma,\gamma)}\gamma
=\gamma+p\d\in R^{in}.$$
Thus, we can prove that
	\begin{align}\label{3.111}
		&\gamma+p\delta \in R^{in} \Leftrightarrow -\gamma+p\delta \in R^{in},
	\end{align}
	for all $p\in \Bbb Z$. Particularly, 
	$$\gamma+\delta \in R^{in} \Leftrightarrow -\gamma+\delta \in R^{in}.$$
 On the other hand, 
	if $\pm \gamma +2\delta\in R^{ln}$, then
	by Theorem \ref{3.8}
	$$\gamma+3\delta=\underbrace{-\gamma-\delta}_{\in R^{ln}}+\underbrace{2(\gamma+2\delta)}_{\in R^{ln}}\in R^{ln},$$
and so by Lemma \ref{3.2}(ii)
	$$-\gamma+3\delta=\underbrace{-\gamma+\delta}_{\in R^{in}}+\underbrace{\gamma+\delta}_{\in R^{in}}+\underbrace{(-\gamma+\delta)}_{\in R^{in}}\in \frak C_M\cap R_{re}=R^{in}.$$
	This contradicts (\ref{3.111}). Hence, $\pm \gamma +2\delta\in R^{in}$.
		Now, assume that $n\in \Bbb Z^{\geq 1}$. Then  by Lemma \ref{3.2}(ii)
		$$\pm\gamma +2n\delta=\underbrace{\pm\gamma+2\delta}_{\in R^{in}}+\underbrace{(n-1)(\gamma+\delta)}_{\in R^{in}}+\underbrace{(n-1)(-\gamma+\delta)}_{\in R^{in}}\in \frak C_M\cap R_{re}^\times=R^{in}.$$
			Furthermore,	by Lemma \ref{3.2}(ii)
	$$\pm\gamma+(1+2n)\delta=\underbrace{\pm\gamma+\delta}_{\in R^{in}}
	+\underbrace{n(\gamma+\delta)}_{\in R^{in}}+\underbrace{n(-\gamma+\delta)}_{\in R^{in}}\in \frak C_M\cap R_{re}=R^{in}.$$
	Now, assume that $n\in \Bbb Z^{\leq 0}$. If $\pm \gamma +2n\delta\in R^{in}$, then by Lemma \ref{3.2}(ii)
	$$\pm\gamma=\underbrace{\pm\gamma+2n\delta}_{\in R^{in}}+\underbrace{(-n(\gamma +\delta))}_{\in R^{in}}+\underbrace{(-n(-\gamma+\delta))}_{\in R^{in}}\in \frak C_M\cap R_{re}=R^{in}.$$
	This contradicts the assumption $\pm\gamma \in R^{ln}$. Hence, 
	$\pm \gamma +2n\delta\in R^{ln}$. Therefore, we have showed that
	\begin{align}
		&\underbrace{\cdots,~ \gamma-2\delta,~ \gamma -\delta,~ \gamma}_{\in R^{ln}}~,~ \underbrace{\gamma +\delta, ~\gamma +2\delta,~ \cdots}_{\in R^{in}},\\
		&\underbrace{\cdots,~ -\gamma-2\delta,~ -\gamma -\delta,~ -\gamma}_{\in R^{ln}}~,~ \underbrace{-\gamma +\delta, ~-\gamma +2\delta,~ \cdots}_{\in R^{in}} \quad ({\rm accutully}~ t=1).
	\end{align}
	This means that (iii) holds.
	
	{\bf Case 2.} $-\gamma \in R^{in}$ and $-\gamma -\delta \in R^{in}$.\\	
	Assume that $\gamma-n\delta\in R^{in}$ for some $n\in\Bbb Z^{>0}$. Then by Lemma \ref{3.2}(ii) and $(*)$
	$$\gamma=\underbrace{\gamma-n\delta}_{\in R^{in}}+\underbrace{n(-\gamma)}_{\in R^{in}}+\underbrace{n(\gamma+\delta)}_{\in R^{in}}\in \frak C_M\cap R_{re}=R^{in}.$$
	This contradicts the assumption $\gamma \in R^{ln}$, see ($*$). Thus 
	$\gamma-n\delta\in R^{ln}$ for all $n\in\Bbb Z^{\geq 0}$.
Based on $(*)$, we have $\gamma +\delta\in R^{in}$. 
So  by Lemma \ref{3.7},
$$r_{\gamma+\delta}(\gamma)=\gamma-\frac{2(\gamma,\gamma+\delta)}{(\gamma+\delta,\gamma+\delta)}(\gamma+\delta)=\gamma-2(\gamma+\delta)=-\gamma -2\delta$$
belongs to $R^{ln}$, because $\gamma \in R^{ln}$. 
If $-\gamma-n\d\in R^{in}$ for $n\geq 3$, then consider the following equation
$$-\gamma-2\delta=-\gamma-n\delta+(n-2)(-\gamma)+(n-2)(\gamma+\delta).$$
By Lemma \ref{3.2}(ii) and $(*)$, the right hand side of the equation belongs to $R^{in}$, while the left hand side of the equation belongs to $R^{ln}$, a contradiction. Thus, $-\gamma-n\delta\in R^{ln}$ for all $n\in\Bbb Z^{\geq 3}$.\\
	Now, assume that $n\in \Bbb Z^{\geq 0}$. Then by Lemma \ref{3.2}(ii) and $(*)$
	\begin{align*}
		&-\gamma+n\delta=\underbrace{(n+1)(-\gamma)}_{\in R^{in}}+\underbrace{n(\gamma+\delta)}_{\in R^{in}}\in R^{in},\\
		&-\gamma+(n+2)\delta=\underbrace{(n+1)(-\gamma)}_{\in R^{in}}+\underbrace{(n+2)(\gamma+\delta)}_{\in R^{in}}\in R^{in}.
	\end{align*}
 Note that by Lemma \ref{3.2}(ii), we have 
 $$\gamma+n\delta=\underbrace{(n-1)(-\gamma)}_{\in R^{in}}+\underbrace{n(\gamma+\delta)}_{\in R^{in}}\in R^{in}$$
 in which $n\geq 1$.
	Therefore, we show that
	\begin{align}
		&\underbrace{\cdots,~ \gamma-2\delta,~ \gamma -\delta,~ \gamma}_{\in R^{ln}}~,~ \underbrace{\gamma +\delta, ~\gamma +2\delta,~ \cdots}_{\in R^{in}},\\
		&\underbrace{\cdots,~ -\gamma-2\delta}_{\in R^{ln}}~,~ \underbrace{ -\gamma -\delta,~ -\gamma,~-\gamma +\delta, ~-\gamma +2\delta,~ \cdots}_{\in R^{in}} \quad ({\rm accutully}~~t=-1).
	\end{align}
	This means that (iii) holds.
	
	{\bf Case 3.} $-\gamma \in R^{in}$ and $-\gamma -\delta \in R^{ln}$.\\	
	Let $n\in \Bbb Z^{\geq 0}$. Then for $n\geq 0$, we have
	\begin{align*}
		&-\gamma+n\delta=\underbrace{(n+1)(-\gamma)}_{\in R^{in}}+\underbrace{n(\gamma+\delta)}_{\in R^{in}}\in R^{in},\\
		&\gamma+(n+1)\delta=\underbrace{n(-\gamma)}_{\in R^{in}}+\underbrace{(n+1)(\gamma+\delta)}_{\in R^{in}}\in R^{in}.
	\end{align*}
	Assume that $\gamma -k\delta\in R^{in}$ for some $k\in \Bbb Z ^{\geq 1}$. Since $-\gamma+k\d=-(\gamma-k\delta)\in R^{in}$ (above), we have 
	$$-\gamma+2k\delta=\gamma-2\gamma+2k\delta=\gamma-\frac{2(\gamma-k\delta,\gamma)}{(\gamma , \gamma)}(\gamma-k\delta)=r_{\gamma-k\delta}(\gamma)\in R^{ln}$$
	by Lemma \ref{3.7}.
	However, according to the above, the left side belongs to $R^{in}$, which is a contradiction. Thus, $\gamma -n\delta\in R^{ln}$ for all $n\in \Bbb Z ^{\geq 0}$. On the other hand, by using the above and Theorem \ref{3.8}
	$$-\gamma-(n+1)\delta=(\gamma-(n-1)\delta)+2(-\gamma-\delta)\in {R^{ln}}+2R^{ln}\subseteq R^{ln},$$
	for all $n\in \Bbb Z ^{\geq 1}$. Therefore, we show that
	\begin{align}
		&\underbrace{\cdots,~ \gamma-2\delta,~ \gamma -\delta,~ \gamma}_{\in R^{ln}}~,~ \underbrace{\gamma +\delta, ~\gamma +2\delta,~ \cdots}_{\in R^{in}},\\
		&\underbrace{\cdots,~ -\gamma-2\delta,~-\gamma -\delta}_{\in R^{ln}}~,~ \underbrace{-\gamma,~-\gamma +\delta, ~-\gamma +2\delta,~ \cdots}_{\in R^{in}} \quad ({\rm accutully}~~t=0).
	\end{align}
	This means that (iii) holds.\\	
	Regarding (iv), in $(**)$, we have
	$$\eta:=\gamma+\delta\in R^{ln},~~\eta+(-\delta)=\gamma\in R^{in}.$$
	This implies that $(**)$ turns to $(*)$, and we get (iv). The proof is complete.
\end{proof}

The above theorem allows us to choose a name for a root according to its behavior correspondence action. The real  
\begin{itemize}
	\item [-] root  $\a$ is {\it fully locally nilpotent}, if 
	$\alpha+\Bbb Z\delta \subseteq R^{ln}$,
	\item [-] root $\a$ is {\it full-injective}, if 
	$\alpha+\Bbb Z\delta \subseteq R^{in}$,
		\item [-] roots $\pm \a$ are {\it up-nilpotent hybrid}, if there is positive integer $m$ such that 
		$$\pm \a +\Bbb Z^{\geq m}\d\subseteq R^{ln}~~{\rm and} ~~\pm \a +\Bbb Z^{\leq -m}\d\subseteq R^{in},$$ 
			\item [-] roots $\pm \a$ are {\it down-nilpotent hybrid}, if there is positive integer $m$ such that 
			$$\pm \a +\Bbb Z^{\geq m}\d\subseteq R^{in}~~{\rm and} ~~\pm \a +\Bbb Z^{\leq -m}\d\subseteq R^{ln}.$$ 
\end{itemize}
A real root $\a$ is said to be {\it hybrid} if $\a$ is either down-nilpotent hybrid or up-nilpotent hybrid, otherwise it is called {\it tight}.
A subset $T$ of $R$ is called a {\it full-locally nilpotent} (resp. {\it full-injective}, {\it down-nilpotent hybrid},
{\it up-nilpotent hybrid}, {\it hybrid, tight}) {\it set}, if all elements of $T_{re}:=  T \cap R_{re}$ are full-locally nilpotent (resp.
full-injective, down-nilpotent hybrid, up-nilpotent hybrid, hybrid, tight).

\begin{remark}\label{3.10}
  Suppose that $\frak g$ is a subalgebra of $\frak L$ containing $\frak H$ with corresponding root system $S$ and $M$
is an $\frak H$-weight $\frak g$-module having shadow. If 
 for each $\a\in S^\times$, $S$ satisfies the condition $\a+\Bbb Z \d\subseteq S$, then by a straightforward check, Theorem \ref{3.9} holds, that is, each real root $\b\in S$ satisfies one of the following:
 \begin{itemize}
		\item [(i)] $\beta +\Bbb Z\delta \subseteq S^{ln}$.
		\item [(ii)] $\beta +\Bbb Z\delta \subseteq S^{in}$.
		\item [(iii)] there exist $m\in \Bbb Z$ and $t\in \{-1,0,1\}$ such that for $\gamma:=\beta +m\delta$, 
		\begin{align*}
			&\gamma +\Bbb Z^{\geq 1}\delta\subseteq S^{in},~~\gamma +\Bbb Z^{\leq 0}\delta \subseteq S^{ln}\\
			&-\gamma +\Bbb Z^{\geq t}\delta \subseteq S^{in},~~-\gamma +\Bbb Z^{\leq t-1}\delta \subseteq S^{ln}
		\end{align*}
		\item [(iv)] there exist $m\in \Bbb Z$ and $t\in \{-1,0,1\}$ such that for $\eta:=\beta +m\delta$, 
		\begin{align*}
			&\eta +\Bbb Z^{\leq -1}\delta \subseteq S^{in},~~	\eta +\Bbb Z^{\geq 0}\delta \subseteq S^{ln}\\
			&-\eta +\Bbb Z^{\leq -t}\delta \subseteq S^{in},~~-\eta +\Bbb Z^{\geq 1-t}\delta\subseteq S^{ln}.
		\end{align*}
	\end{itemize}
 Hence, we are again able to define a full-locally nilpotent (resp. full-injective, down-nilpotent hybrid,
up-nilpotent hybrid, hybrid) subset of $S$ in exactly the same way as above.
\end{remark}

We recall that since the form on $\frak L$ restricted to $\frak H$ is nondegenerate, symmetric, and bilinear, we can deduce that for each $\a\in \frak H^*$, there is a unique 
$t_\a\in \frak H$ with $\a(h)=(t_\a,h)$ for all $h\in \frak H$.

To prove the next result, we need to recall some materials from the literature:
\begin{itemize}
    \item [(i)] An {\it extended affine Lie algebra} is a triple $(\frak K, (-, -), \frak h)$ consisting of two Lie algebras $0\neq \frak h\subseteq \frak K$ and a symmetric bilinear form 
    $(-,-):\frak K\times \frak K\longrightarrow \Bbb C$ such that
\begin{itemize}
    \item [(a)] the form $(-, -)$ is nondegenerate and invariant,
   \item [(b)] $\frak h$ is finite dimensional, abelian, and self-centralizing. Moreover,
$${\rm ad}_h:\frak K\longrightarrow \frak K, ~~{\rm ad}_h(x)=[h,x]$$
    is diagonalizable for all $h\in \frak h$ (this leads us to the root system $T$, and we are able to define $T_{re}:=\{\a\in T\mid (\a,\a)\neq 0\}$ and $T_{\circ}:=\{\a\in T\mid (\a,\a)=0\}\cup \{0\}$),
\item [(c)]    ${\rm ad}_x:\frak K\longrightarrow \frak K$ is locally nilpotent for all $x\in \frak K^\a$ and $\a\in T_{re}$,
\item [(d)] $T_{re}$ is irreducible, 
\item [(e)] the subgroup $G:={\rm span}_{\Bbb Z}(T_\circ)$ in the group $(\frak h^*,+)$ is a free abelian group of finite rank ($G\cong \Bbb Z^n$ as a group for some $n\in \Bbb N\cup\{0\}$),\\
\item [(f)]
$T_\circ$ has no isolated roots, that is, given $\b\in T_\circ$ there exists $\a\in T_{re}$ such that $\a+\b\in T$.
\end{itemize}
 \item [(ii)] Let $\frak K$ be an extended affine Lie algebra with root system $T$. The Lie algebra $\frak K$ is called {\it tame} if $\frak K_c^\perp\subseteq \frak K_c$ in which 
 $$\frak K_c:=\langle \frak K^\a\mid \a\in T_{re}\rangle_{\rm subalgebra}~~{\rm and}~~\frak K_c^\perp:=\{x\in \frak K\mid (x,\frak K_c)=\{0\}\}.$$
\item [(iii)] Let $\frak K$ be an extended affine Lie algebra. The rank of the free abelian group $G$ in axiom (e) is called the {\it nullity} of $\frak K$.
\end{itemize}

\begin{The}[\cite{ABGP}, Theorem 2.23]\label{3.11}
 A Lie algebra $\frak K$ (over $\Bbb C$) is isomorphic to an affine (Kac-Moody) Lie algebra
if and only if $\frak K$ is isomorphic to a tame extended affine Lie algebra of nullity one.    \end{The}

\begin{Pro}\label{3.12}
Suppose that $S$ is a nonempty symmetric closed subset of the root system $R$ of $\frak L$ with $S_{re} :=
R_{re} \cap S\neq \emptyset$. Set
$$S_0 := S \cap R_0,~~{\rm and}~~\dot S:=\{\dot\a\in \dot R~|~(\dot \a +\Bbb Z \d)\cap S_0\not =\emptyset\}.$$
If for $0\neq \dot \a\in \dot S$, $(\dot\a+\Bbb Z \d)\cap R_0\subseteq S$, then the following statements hold:
\begin{itemize}
    \item [(i)] $\dot S$ is a finite root system so that we may assume with irreducible components $\dot S (1), \cdots, \dot S(k)$ of $\dot S$.
\item[(ii)] $\frak L(i):=\frak H +\sum_{\a\in S(i)^\times} \frak L ^\a +\sum_{\a,\b\in S(i)^\times} [\frak L^\a,\frak L^\b] $ is an affine Lie algebra (up to a central space) in which $S(i):=(\dot S(i) +\Bbb Z\d) \cap R_0$ for $1\leq i\leq k$.  
\end{itemize}
\end{Pro}

\begin{proof}
  (i)   Choose an element $\a=\dot\a+m\d\in S_{re}$ for some $m\in \Bbb Z$. Note that $\dot\a\neq 0$. If $\a\in R_0$, then $\dot\a\in \dot S$. If not, according to Table \ref{Table 2}, Table \ref{Table 3} and Table \ref{Table 4}, we have $\a\in R_1\cap R_{re}\subseteq R_0$. Since $S$ is closed and $2(R_1\cap R_{re})\subseteq R_0$, we have $2\a \in S\cap R_0$. Hence, we get $\dot S\neq \{0\}$. 
  Being a finite root system $\dot S$ is easily verified.\\
  (ii) We prove the statement for type $A(n,n)^{(1)}$ and the proof can be repeated for type $X\neq A(n,n)^{(1)}$ with a similar process. So, we use the notations related to type $A(n,n)^{(1)}$ in the proof process freely.\\
 Assume that $(rc_2+sd_2,\Bbb Cc_2\oplus \Bbb Cd_2)=\{0\}$  where $r,s\in \Bbb C$. 
  Since $(rc_2+sd_2,c_2)=0$ and $(rc_2+sd_2,d_2)=0$, we have $r=s=0$.
  Thus, the restriction of form on $\Bbb Cc_2\oplus \Bbb Cd_2$ is nondegenerate. 
According to (i), $\dot S(i)$ is an irreducible finite root system. If 
$\sum_{\dot\a\in \dot S(i)}\frak L^{\dot \a}$ is a simple Lie algebra, then $\sum_{\dot\a\in \dot S(i)}\Bbb Ct_{\dot\a}$ is contained in the Cartan subalgebra. Hence, form $(-,-)$ is 
  nondegenerate on $\sum_{\dot\a\in \dot S(i)}\Bbb Ct_{\dot\a}$.
Otherwise, the Lie algebra $\sum_{\dot\a\in \dot S(i)}\frak L^{\dot \a}$ is of type $BC(n)$, and it contains $\sum_{\dot\a\in B(n)}\frak L^{\dot \a}$. Hence, 
$\sum_{\dot\a\in \dot S(i)}\Bbb Ct_{\dot\a}\subseteq \sum_{\dot\a\in B(n)}\frak L^{\dot \a}$. This leads us that the form $\fm$ is nondegenerate on $\sum_{\dot\a\in \dot S(i)}\Bbb Ct_{\dot\a}$. Recalling (\ref{form Ann}), we have 
\begin{align}\label{hhhh}
  &  (\Bbb Cc_2\oplus \Bbb Cd_2, \sum_{\dot\a\in \dot S(i)^\times}\Bbb Ct_{\dot\a})=0.
\end{align}
Let $\a=\dot \a +k\d \in S(i)$ for some $k\in \Bbb Z$. Note that 
$t_\a=r_1c_1+r_2c_2+s_1d_1+s_2d_2+h$ (we identify $h\otimes 1$ by $h$).
This implies that 
$$
\a(u)=(t_\a ,u)=\begin{cases}
    s_1\qquad \qquad & u=c_1\\
s_2\qquad\qquad &u=c_2\\
r_1\qquad\qquad &u=d_1\\
r_2\qquad\qquad & u=d_2\\
(h,h')\qquad &u=h
\end{cases}
$$
and 
$$
\a(u)=(\dot\a+k\d)(u)=\begin{cases}
    0\qquad \qquad & u=c_1\\
0\qquad\qquad &u=c_2\\
0\qquad\qquad &u=d_1\\
k\qquad\qquad & u=d_2\\
\dot\a(h')\qquad &u=h
\end{cases}.
$$
Thus
$$
\begin{cases}
   s_1=0\\
   s_2=0\\
   r_1=0\\
   r_2=k\\
   (h,h')=\dot\a(h')=(t_{\dot\a},h')\Rightarrow h=t_{\dot\a}
\end{cases}
$$
So, $t_\a=r_2c_2+h=kc_2+t_{\dot\a}$.
Therefore, actually, with the help of (\ref{hhhh}), we have shown that
$$\Bbb Cc_2\oplus \Bbb Cd_2 + \sum_{\a\in \dot S(i)}\Bbb Ct_{\dot\a}= \Bbb Cc_2\oplus \Bbb Cd_2\oplus \sum_{\dot\a\in \dot S(i)^\times}\Bbb Ct_{\dot\a}.$$
Define 
$$\frak H(i) :=\Bbb Cc_2\oplus \Bbb Cd_2\oplus \sum_{\dot\a\in \dot S(i)^\times}\Bbb Ct_{\dot\a}.$$
It follows that the form is nondegenerate on $\frak H(i)$ and so, there is an orthogonal complement $T_i$ for $\frak H(i)$ in $\frak H$. Now, for $\a\in S(i)$, $x\in \frak L^\a$ and $h\in T_i,$ we have
\begin{align*}
[h,x]=\a(h)x=(t_\a,h)x\in(\frak H(i),T_i)x=\{0\}
,\end{align*}
that is, $T_i$ is contained in the center of $\frak L(i).$ 
Define 
\[\fk(i):=\frak H(i)+\summ{\alpha \in S{(i)}^\times}\frak{L}^{\alpha}+\summ{\alpha,\b \in S{(i)}^\times}[\frak L^\a,\frak L^{\b}].\] 
We show that 
$$\fk(i)=\frak H(i)\op\Bigop{\alpha \in S{(i)}^\times}\frak{L}^{\alpha}\op\Bigop{0\neq k\in\bbbz}(\sum_{ \dot\a\in \dot S(i)^\times}\sum_{t\in\bbbz}[\frak L^{\dot\a+t\d},\frak L^{-\dot\a+(k-t)\d}]).$$
Clearly, $\frak H(i)\op\Bigop{\alpha \in S{(i)}^\times}\frak{L}^{\alpha}\op\Bigop{0\neq k\in\bbbz}(\sum_{ \dot\a\in \dot S(i)^\times}\sum_{t\in\bbbz}[\frak L^{\dot\a+t\d},\frak L^{-\dot\a+(k-t)\d}])\subseteq \fk(i)$. Conversely, assume that 
$u\in [\frak L^\a,\frak L^{\b}]\subseteq \frak L^{\a+\b}$ for some $\alpha,\b \in S{(i)}^\times$.
Let $\a=\dot\a+m\d$ and $\b=\dot\b+n\d$ for some $m,n\in \Bbb Z$. 
 Now, if 
$\dot\a+\dot \b\neq 0$, then $\a+\b\in S(i)^\times$, because $S(i)$ is closed.
So, $\frak L^{\a+\b}\subseteq \Bigop{\alpha \in S{(i)}^\times}\frak{L}^{\alpha}$.
This shows that 
$$u\in \frak H(i)\op\Bigop{\alpha \in S{(i)}^\times}\frak{L}^{\alpha}\op\Bigop{0\neq k\in\bbbz}(\sum_{ \dot\a\in \dot S(i)^\times}\sum_{t\in\bbbz}[\frak L^{\dot\a+t\d},\frak L^{-\dot\a+(k-t)\d}]) .$$
Otherwise, if $\dot\a+\dot \b= 0$, then there is $k\in \Bbb Z$ such that $n=k-m$ and 
$$u\in [\frak L^\a,\frak L^{\b}]=[\frak L^{\dot\a+m\d},\frak L^{-\dot \a+n\d}]=
[\frak L^{\dot\a+m\d},\frak L^{-\dot \a+(k-m)\d}]\subseteq \frak H(i)\op\Bigop{\alpha \in S{(i)}^\times}\frak{L}^{\alpha}\op\Bigop{0\neq k\in\bbbz}(\sum_{ \dot\a\in \dot S(i)^\times}\sum_{t\in\bbbz}[\frak L^{\dot\a+t\d},\frak L^{-\dot\a+(k-t)\d}]).$$
Therefore, the desired equality is achieved.\\
For each element $\a=\dot\a+k\d\in S(i)$ where $k\in \Bbb Z$, we have
$$\a(T_i)=(t_\a, T_i)=(t_{\dot\a}+kd_2, T_i)\subseteq (\frak H(i), T_i)=\{0\}.$$
This forces that if $\a\in S(i)$ and $\b\in S(i)$ are distinct on $\frak H$, then they are distinct on $\frak H(i)$. Hence, the decomposition  
\[\fk(i)=\frak H(i)\op\Bigop{\alpha \in S{(i)}^\times}\frak{L}^{\alpha}\op\Bigop{0\neq k\in\bbbz}(\sum_{ \dot\a\in \dot S(i)^\times}\sum_{t\in\bbbz}[\frak L^{\dot\a+t\d},\frak L^{-\dot\a+(k-t)\d}])\] coincides with  the root space decomposition of $\fk(i)$ with respect to $\frak H(i).$\\
We claim that $\fk(i)$ is an affine Lie algebra. To prove this, we want to apply 
Theorem \ref{3.11}. Therefore, we will provide the details in the following steps:
\begin{itemize}
    \item [\bf S1.] The form $(-, -)$ is nondegenerate and invariant on $\frak k(i)$.\\
Clearly, the invariant is inherited from $\frak L$ to $\frak k(i)$.
   Assume that $x\in \frak k(i)$ and $(x,\frak k(i))=0$. We can write 
   $$x=\underbrace{r_1t_\a+r_2c_2+r_3d_2}_{\in \frak H(i)}+\underbrace{r_4y_{\b}}_{\in \sum_{ \a\in {S(i)^\times}}\frak L^\a}+\underbrace{r_5h\otimes t^l}_{\in \sum_{ \dot\a\in \dot S(i)^\times}\sum_{m,n\in\bbbz}[\frak L^{\dot\a+m\d},\frak L^{-\dot\a+n\d}]}.$$
As $(x,c_2)=(x,d_2)=0$, we get $r_2=r_3=0$. Thus 
$$x=r_1t_\a+r_4y_\b+r_5h\otimes t^l.$$
We know that $t_\a$ and $y_\b$ identify by $h'\otimes t^0$ and $y_{\dot\b}\otimes t^m$ respectively. It follows from 
$(x,h'\otimes t^0)=0$ that $r_1=0$. By choosing the appropriate members of $\frak k(i)$ and repeating the above process, it is concluded that $r_4=r_5=0$. Hence, $x=0$.
    \item [\bf S2.] $\frak H(i)$ is finite dimensional, abelian and self-centralizing. Furthermore, ${\rm ad}_h:\frak k(i)\longrightarrow \frak k(i)$  is diagonalizable for all $h\in \frak H(i)$.\\
Clearly, finite dimensional and abelian are inherited from $\frak H$ to $\frak H(i)$. We have 
$$\frak H(i)\subseteq \{x\in \frak k(i) \mid [x,\frak H(i)]=\{0\}\}.$$
Conversely, assume that $x\in \frak k(i)$ and $[x,\frak H(i)]=\{0\}$. Then, 
$x\in \frak L^\a$ with $\a\neq 0$. Thus for all $h\in \frak H(i)$, we get
$$\a(h)x=[h,x]=-[x,h]=0.$$
It follows that ${\rm ad_h}$ is zero, a contradiction. For the last statement, we just note that ${\rm ad}_h$ is diagonalizable for all $h\in \frak H$.
    \item [\bf S3.]  ${\rm ad}_x:\frak k(i)\longrightarrow \frak k(i)$ is locally nilpotent for all $x\in \frak k^\a$ and $\a\in S(i)_{re}$.\\
For $y\in \frak L^\b$, we have
$$({\rm ad}_x)^l(y)=\underbrace{[x[x,\cdots ,\underbrace{x,\underbrace{[x,y]}_{\in \frak L^{\a+\b}}]}_{\in\frak L^{2\a+\b}}\cdots]}_{\in \frak L^{l\a+\b}}.$$
This is while $l\a+\b$ is not root for sufficiently large $l$.
    \item [\bf S4.]  $S(i)_{re}$ is irreducible. \\
Assume that $S(i)_{re}=A\cup B$ with $(A,B)=0$, $A\neq \emptyset$ and $B\neq\emptyset$ . Then $\dot S(i)=\dot A\cup \dot B$ with 
$(\dot A,\dot B)=0$. Since $\dot S(i)$ is irreducible, it follows that either $\dot A=\emptyset$ or $\dot B=\emptyset$. This leads us to either $A=\emptyset$ or $ B=\emptyset$, a contradiction.
    \item [\bf S5.]  $S(i)_\circ$ has no isolated roots and ${\rm span}_{\Bbb Z}(S(i)_\circ)\cong \Bbb Z$.\\
Note that $S(i)_\circ=\Bbb Z\d$ and $(\dot\a+\Bbb Z\d)\cap R_0\subseteq S$ for $0\neq \a\in \dot S$. So, we have ${\rm span}_{\Bbb Z}(S(i)_\circ)=\Bbb Z\d\cong \Bbb Z$.
    \item [\bf S6.] $(\fk(i))_c=\sum_{ \a\in {S(i)^\times}}\frak L^\a\op\sum_{ \dot\a\in \dot 
S(i)^\times}\sum_{m,n\in\bbbz}
[\frak L^{\dot\a+m\d},\frak L^{-\dot\a+n\d}]$.\\
Note that $R_0\setminus \{0\}\subseteq R_{re}$ according to Table \ref{Table 3} and Table \ref{Table 4}. Thus, 
if $0\neq \a\in S(i)^\times \subseteq R_0$, then $(\a,\a)\neq 0$, and so
$\sum_{ \a\in {S(i)^\times}}\frak L^\a\subseteq (\fk(i))_c$.
Assume that $\dot \a, -\dot\a \in \dot S(i)^\times$. Then $\dot\a+m\d, -\dot\a+n\d \in S(i)^\times\subseteq R_0$ where $m,n\in \Bbb Z$. Clearly, 
$\frak L^{\dot\a+m\d}, \frak L^{-\dot\a+n\d}\subseteq (\fk(i))_c$. As $(\fk(i))_c$ is algebra, 
it follows that $[\frak L^{\dot\a+m\d}, \frak L^{-\dot\a+n\d}]\subseteq (\fk(i))_c$.
So far we have shown that
$$\sum_{ \a\in {S(i)^\times}}\frak L^\a\op\sum_{ \dot\a\in \dot S(i)^\times}\sum_{m,n\in\bbbz}[\frak L^{\dot\a+m\d},\frak L^{-\dot\a+n\d}]\subseteq(\fk(i))_c.$$
To achieve equality, it is enough to note that the materials of $(\fk(i))_c$ are either materials in the form $\frak L^\a$ or a summation of materials in the form of 
$[\frak L^\a,\frak L^\b] \subseteq \frak L^{\a+\b}$.
    \item [\bf S7.]  $\fk(i)=(\fk(i))_c\oplus \Bbb Cc_2$.\\
We note that $\frak H(i) :=\Bbb Cc_2\oplus \Bbb Cd_2\oplus \sum_{\dot\a\in \dot S(i)^\times}\Bbb Ct_{\dot\a}$ and 
$$\fk(i)=\frak H(i)\op\Bigop{\alpha \in S{(i)}^\times}\frak{L}^{\alpha}\op\Bigop{0\neq k\in\bbbz}(\sum_{ \dot\a\in \dot S(i)^\times}\sum_{t\in\bbbz}[\frak L^{\dot\a+t\d},\frak L^{-\dot\a+(k-t)\d}]).$$
Since $[\frak L^{\dot \a +m\d},\frak L^{-\dot\a+l\d}]\subseteq \frak L^{(m+l)\d}=\frak H(i)\otimes t^{m+l}$, we have
\begin{align*}
    [h_1\otimes t^m, h_2\otimes t^l]&=[h_1,h_2]\otimes t^{m+l}+\omega(h_1\otimes t^m,h_2\otimes t^l)\\
    &=[h_1,h_2]\otimes t^{m+l}+\frac{1}{n+1}\delta_{m+l,0}(\delta^*(h_1),h_2)c_1+m\delta_{m+l,0}(h_1,h_2)c_2\\
    &=[h_1,h_2]\otimes t^{m+l}+m\delta_{m+l,0}(h_1,h_2)c_2.
\end{align*}
Moreover, 
$$[d_2,h\otimes t^m]=mh\otimes t^m.$$
With a comparison, it is clear that from $\frak H(i)$, only $\Bbb Cc_2$ is not contained in $(\fk(i))_c$. Hence, $\fk(i)=(\fk(i))_c\oplus \Bbb Cc_2$.
    \item [\bf S8.]  $\frak k(i)$ is tame.\\
 Assume that $x\in (\frak k(i)_c)^\perp$. Then $x=y+rd_2$ when $y\in (\frak k(i))c$ and $r\in \Bbb C$. If $r\neq 0$, then 
 $$0=(x,c_2)=(y+rd_2,c_2)=r(d_2,c_2=r,$$
which is a contradiction. This implies that $x\in (\frak k(i))c$, as desired.
  \end{itemize}
We note that the nullity of $\frak k(i)$ is one, according to S5. 
Therefore, 
$$\frak L(i)=\frak H +\sum_{\a\in S(i)^\times} \frak L ^\a +\sum_{\a,\b\in S(i)^\times} [\frak L^\a,\frak L^\b]=Z\oplus \frak k(i)$$ 
is an affine Lie algebra up to the central space $Z$.
\end{proof}

Keep the notation of Proposition \ref{3.12}. For $i\in \{1,\cdots, k\}$, define

$$P_i:=\begin{cases}
	S(i)^{ln}\cup-S(i)^{in}\cup\Bbb Z^{\geq 0}\delta  & {\rm if}~S(i)~{\rm is~up-nilpotent~hybrid}\\
	S(i)^{ln}\cup-S(i)^{in}\cup\Bbb Z^{\leq 0}\delta  & {\rm if}~S(i)~{\rm is~down-nilpotent~hybrid}
\end{cases}.$$


\begin{remark}\label{3.13}
  Keep the notation of Proposition \ref{3.12}. With an imitation of \cite[Remark 5.5]{Yousofzadeh:Finite weight 2020}, we can suppose that there is 
   \begin{itemize}
       \item[($\dagger$)] a base $\Pi_i$ of $S(i)$ such that the set of positive roots of $S(i)$ with respect to $\Pi_i$ is a subset of $P_i$,
        \item [($\ddagger$)] a map 
   $$\zeta : {\rm span}_{\Bbb R}S(i)\longrightarrow \Bbb R$$
   with $\zeta (\d)>0$ and $P_i=S(i)^+\cup S(i)^\circ$.
   \end{itemize}
 \end{remark}

\begin{Lem}\label{3.14}
Keep the notation of Proposition \ref{3.12}.
Suppose that $i\in \{1,\cdots ,k\}$ and $S(i)=S(i)^+\cup S(i)^\circ\cup S(i)^-$ is a  triangular decomposition for $S(i)$ with corresponding linear functional $\zeta$ such that 
	$\zeta (\delta)>0$, $S(i)^+\cap S_{re}\subseteq S^{ln}$ and 
	$S(i)^-\cap S_{re}\subseteq S^{in}$. Assume that $W$ is an $\frak L_0$-submodule of $M$. Then there is a positive integer $p$ and $\lambda \in {\rm supp}(W)$ with $(\lambda +\Bbb Z^{>0}p\delta)\cap {\rm supp}(W)=\emptyset$.
\end{Lem}
\begin{proof}
It is obtained with appropriate changes in proof \cite[Proposition 2.17]{D-G} or repeat the proof of \cite[Lemma 5.1]{Yousofzadeh:Finite weight 2020}.
\end{proof}

 \begin{Pro}\label{3.15}
 Suppose that $\frak k$ is a subalgebra of $\frak L$ containing $\frak H$ with corresponding root system $T$, and $M$ is a $\frak k$-module having shadow. If 
 \begin{itemize}
     \item [1)]$S$ is a symmetric closed nonempty subset
of $T$,
\item [2)] $S_{re}\neq \emptyset$ and $(S+\Bbb Z\d)\cap R_0\subseteq S$, 
\item [3)]$S$ is hybrid,
 \end{itemize}
then either all real roots of $S$ are
up-nilpotent hybrid or all are down-nilpotent hybrid. 
  \end{Pro}
\begin{proof}
 We know from Table \ref{Table 2} and Table \ref{Table 4} that if $\a \in R_{re}\cap R_1$, then $2\a \in R_{re}\cap R_0$. Thus, if 
 $\a \in R_{re}\cap R_1$, 
 then by Theorem \ref{3.8}, $\a\in R^{ln}(V)$ if and only if $2\a\in R^{ln}(V)$.
 So,  we can assume $S\sub R_0$ without loss of generality.
 By Proposition \ref{3.13}, we have  $\dot{S}=\{\dot{\alpha}\in \dot R \mid (\dot\a+\bbbz\d)\cap S\neq \emptyset\}$ is a finite root system with irreducible components  $\dot S(1),\ldots, \dot S(k)$.
By the assumption and Remark \ref{3.10} (see Theorem \ref{3.9}), we can write 
$$\hbox{\small $S_{re}= \underbrace{\{\a\in S_{re}\mid \exists N, ~\a+\bbbz^{\geq N}\d\sub  R^{ln}\}}_{:=K_1}\uplus  \underbrace{\{\a\in S_{re}\mid \exists N, ~\a+\bbbz^{\geq N}\d\sub  R^{in}\}}_{:=K_2}$ }$$ in which ``~$\uplus$~'' indicates disjoint union.
In the following two cases, we show that $K_1\cup \{0\}$ and $K_2\cup \{0\}$ are symmetric closed subsets of $S_{re}\cup \{0\}$:
\begin{itemize}
    \item [\bf C1.]  Let $\a,\b \in K_1\cup \{0\}$. Then, we prove that $\a+\b\in K_1\cup \{0\}$. For some $N_1$ and $N_2$
$$\a+\Bbb Z^{\geq N_1}\d\subseteq R^{ln},~~{\rm and}~~\b+\Bbb Z^{\geq N_1}\d\subseteq R^{ln}.$$
Set $N=N_1+N_2$. For $n>N$, by Theorem \ref{3.8}, we have
$$\a+\b+n\d=\underbrace{\a+N_1\d}_{\in R^{ln}}+\underbrace{\b+(n-N_1)\d}_{\in R^{ln}}\in R^{ln}.$$
Note that $\a+\b+n\d\in R$ , because $S$ is closed and $\a+\b\in S\subseteq R_0$. 
On the other hand, by Remark \ref{3.10} (see Theorem \ref{3.9}),  $-\a+\Bbb Z ^{\geq N_3}\subseteq R^{ln}$ for some $N_3$. Therefore, $K_1\cup \{0\}$ is a  symmetric closed subset of $S_{re}\cup \{0\}$.
\item [\bf C2.] Let $\a,\b \in K_2\cup \{0\}$. If $\a+\b =0$, then there is nothing to prove. Assume that $\a+\b \neq 0$. Then, for appropriates $n$ and $m$, $\a+n\d,\b+m\d \in S_{re}\cap R^{in}\subseteq \frak{C}_M$. By Lemma \ref{3.2}(ii), 
$\a+\b+(n+m)\d \in \frak C_M$. As $S$ is closed, we get $0\neq \a+\b+(n+m)\d\in S\subseteq R_0$. Hence, $\a+\b+(n+m)\d \in R_{re}$, and so $\a+\b +(n+m)\d\in \frak C_M\cap R_{re}=R^{in}$. This implies that $K_2$ is closed. 
On the other hand, by Remark \ref{3.10} (see Theorem \ref{3.9}),  $-\a+\Bbb Z ^{\geq N}\subseteq R^{in}$ for some $N$. Thus,  $K_2\cup \{0\}$ is a symmetric closed subset of $S_{re}\cup \{0\}$.
\end{itemize}
Since $S$ is symmetric and closed, we get $\dot S$ is symmetric and closed. This implies that $\dot S(1), \cdots , \dot S(k)$ are symmetric and closed.
 A straightforward check shows that $S(1), \cdots , S(k)$ are symmetric and closed.
So, for $1\leq i\leq n,$ $S(i)\cap (K_1\cup \{0\})$ as well as $S(i)\cap (K_2\cup \{0\})$ are symmetric closed subsets of $S(i)_{re}$.
Assume that  $S(i)\cap (K_1\cup \{0\})\neq \emptyset$ and  $S(i)\cap (K_2\cup \{0\})\neq \emptyset$.
Let $\a\in S(i)\cap (K_1\cup \{0\})$ and $\b \in S(i)\cap (K_2\cup \{0\})$. 
If $(\a,\b)\neq 0$, then by \cite[Lemma 3.7]{Yousofzadeh:Extended affine 2016}, either $\a-\b \in R$ or $\a+\b\in R$. 
Suppose that $0\neq \a+\b$. As $S(i)$ is closed and $S\setminus \{0\}\subseteq R_0\setminus \{0\}\subseteq R_{re}$, we get $\a+\b\in S(i)_{re}$, because $S(i)$ is closed.
It follows that $\a+\b\in K_1$ or $\a+\b \in K_2$. If $\a+\b\in K_1$, then 
$\b=(\a+\b)-\a\in K(1)$ (because $K_1$ is closed and symmetric), a contradiction. 
By a similar argument, state $\a+\b \in K_2$ arrives at a contradiction.
If $\a+\b=0$, then $\b=(\a+\b)-\a)\in K_1$, which is a contradiction. 
For the case where \( \alpha - \beta \in R \), we can make a contradiction in a similar way demonstrated above. Thus $(\a,\b)=0$. So, we can find elements $\dot \a$ and $\dot \b$ in $\dot s(i)$ that have the property $(\dot \a,\dot \b)=0$. It follows that $\dot s(i)$ is not irreducible, and this is a contradiction.
Taking everything into account, we have shown that either $S(i)\cap K_1=\emptyset$ or $S(i)\cap K_2=\emptyset$. 
 This leads us to the fact that 
   $S(i)$ is up-nilpotent hybrid or  down-nilpotent hybrid. \\
 We claim that if $k\geq 2$ and $S(i)$ is 
  	up-nilpotent  (resp. down-nilpotent) hybrid for some $i$, then $S(j)$ is up-nilpotent  (resp. down-nilpotent) hybrid for all $1\leq j \leq k$.  
To prove this, fix an $i\in \{1, \cdots , k\}$ and assume that $S(i)$ is an up-nilpotent hybrid (the other case can be proven similarly).
Suppose on the contrary that $j\neq i$ and $S(j)$ is a down-nilpotent hybrid subset. By Remark \ref{3.13}, there exists a functional $f_i$ on $\text{span}_\mathbb{R} S(i)$ with 
$$P_i = S(i)^+ \cup S(i)^\circ = S(i)^{ln}\cup-S(i)^{in}\cup\Bbb Z^{\geq 0}\delta$$ 
and $f_i (\delta) > 0$.
This implies that $S(i)^+\cap S_{re}\subseteq S^{ln}$ and 
$S(i)^-\cap S_{re}\subseteq S^{in}$. So, by Lemma \ref{3.14},
there exist  $p\in \mathbb{Z}$ and $\mu \in \text{supp}(M)$ such that
  	\begin{align}
  		(\mu +\mathbb{Z}^{>0}p\delta)\cap \text{supp}(M)= \emptyset. \label{II}
  	\end{align}
Let $\beta \in S(j)\cap S_{re}$. Since $S(j)$ is down-nilpotent hybrid, we choose $m > 0$ such that
  	$$\pm \beta -np\delta \in S(j)^{ln} \quad \text{and} \quad \pm \beta +np\delta \in S(j)^{in} \quad (n > m).$$ 
 Assume that $\mu+\beta -mp\delta \in \text{supp}(M)$. Then, as $-\beta +2mp\delta \in S(j)^{in}$, we have 
  	$$\mu+mp\delta= (\mu +\beta-mp\delta)-\beta+2mp\delta \in \text{supp}(M)$$
  	which is a contradiction due to \eqref{II}. Hence, 
\begin{equation}\label{L(j)1}
   \frak{L}(j)^{\beta -mp\delta} M^{\mu}=\{0\} 
\end{equation}
where $\frak{L}(j)$ is the affine Lie algebra up to a central space corresponding to $S(j)$.
On the other hand, since $\beta, \beta+2mp\delta \in S(j)\cap S_{re}$, the root string property\footnote{We refer the reader to \cite{Neher} or \cite [Proposition 3.8]{Yousofzadeh:Extended affine 2016} for more details.} for the affine root system $S(j)$ implies that $2mp\delta \in S(j)$.
So by \eqref{II}, we have 
\begin{equation}\label{L(j)2}
\frak{L}(j)^{2mp\delta} M^{\mu}=\{0\}.
\end{equation}
Suppose that $\b=\dot\b+k\d$ for some $k\in \Bbb Z$ and $0\neq x\in \frak{L}(j)^{\dot\b}$. Then, according to $\b$, we can choose $h\in\frak H$ in such a way that $\b(h)=1$\footnote{To find an appropriate $h$ concerning $\b$, use Table \ref{Table 2}, Table \ref{Table 3}, and the definition of $\frak H$ in Section 2.}, and hence
$$[x\otimes t^{l_\b}, h\otimes t^{2mp}]=-[x,h]\otimes t^{l_\b+2mp}=-\b(h)x\otimes t^{l_\b+2mp}=-x\otimes t^{l_\b+2mp}\neq 0.$$
This implies that $[\frak{L}(j)^{\beta-mp\delta}, \frak{L}(j)^{2mp\delta}]\neq 0$.
 By  the fact that 
 $[\frak{L}(j)^{\beta-mp\delta}, \frak{L}(j)^{2mp\delta}]\subseteq \frak{L}(j)^{\beta+mp\delta}$ and \cite[Proposition 2.5]{Yousofzadeh:Extended affine 2019}, we can conclude that $[\frak{L}(j)^{\beta-mp\delta}, \frak{L}(j)^{2mp\delta}]= \frak{L}(j)^{\beta+mp\delta}$.
  	Therefore, by (\ref{L(j)1}) and (\ref {L(j)2}), we have
  	$$\frak{L}(j)^{\beta+mp\delta} M^{\mu}=[\frak{L}(j)^{\beta-mp\delta}, \frak{L}(j)^{2mp\delta}] M^{\mu}=\{0\}.$$
  This contradicts the fact that $\beta +mp\delta \in S(j)^{in}$. Thus, $S(j)$ is up-nilpotent hybrid, and consequently all real roots of $S$ are
up-nilpotent hybrid.
\end{proof}

A subset $P$ of  $R$ is called {\it parabolic}, if 
$$(P+P)\cap R\subseteq P,~~R=P\cup -P.$$

\begin{Pro}\label{3.16}
	Keep the notation of Proposition \ref{3.12}. Suppose that each $S(i)$ is hybrid. Then each $P_i$ is a proper parabolic subset of $S(i)$.
\end{Pro}
\begin{proof}
	Repeat the proof of \cite[Lemma 5.4]{Yousofzadeh:Finite weight 2020}, and use Theorem \ref{3.8}, Proposition \ref{3.3}(ii), Remark \ref{3.10}	(Theorem \ref{3.9}).
\end{proof}

\begin{Pro}\label{3.17}
Keep the notation of Proposition \ref{3.12}.	Suppose that  $S$ is hybrid, and $P:=P_1\cup \cdots \cup P_k$. Then there is a linear functional $\zeta$ on 
	${\rm span}_{\Bbb R}S$ with $\zeta (\d)>0$ such that 
	$$P=\{\alpha \in S~|~\zeta (\alpha)\geq 0\},$$
	in particular, 
	$$\{\alpha\in S\cap R_{re}~|~\zeta (\alpha)>0\}\subseteq R^{ln},~~
	\{\alpha\in S\cap R_{re}~|~\zeta (\alpha)<0\}\subseteq R^{in}.$$	
\end{Pro}

\begin{proof}
By Proposition \ref{3.14}, we can assume that $S$ is up-nilpotent hybrid, and so all $S(i)$ are up-nilpotent hybrid. According to Remark \ref{3.13} (in fact, using affine Lie theory), we can suppose that  $\Pi_i$  is a base  of $S(i)$ as follows:
\begin{align*}
    &\Pi_i =\{\a_{1,i},\a_{2,i},\cdots , \a_{n_i, i}, \d-\theta_i\}\subseteq P_i
\end{align*}
in which each 
\begin{align*}
    &B_i =\{\a_{1,i},\a_{2,i},\cdots , \a_{n_i, i} \}
\end{align*}
is the base of irreducible finite root system $\dot S(i)$ with certain properties,
and $\theta_i$ is the highest root of $\dot S(i)$ with respect to $B_i$. After a reindexing, we may assume that 
\begin{align*}
    &\a_{1,1},\a_{2,1},\cdots , \a_{t_1, 1},\a_{1,2},\a_{2,2}, \cdots, \a_{t_2, 2},\cdots, \a_{1,k},\a_{2,k}, \cdots, \a_{t_k, k}\in P\setminus -P, \\
     &\a_{t_1 +1,1},\a_{t_1+2,1},\cdots , \a_{n_1,1},\a_{t_2 +1,2},\a_{t_2+2,2}, \cdots, \a_{n_2,2},\cdots, \a_{t_k+1, k},\a_{t_k+2,k}, \cdots, \a_{n_k,k}\in P\cap -P.
\end{align*}
 Assume that 
$$\theta_i=\sum_{j=1}^{n_i}r_{j,i}\a_{j,i}$$
where each $r_{j,i}$ is a positive integer, because each $B_i$ is a base. We now proceed in 
the following sequence of cases to define a functional $\zeta$
on ${\rm span}_{\Bbb R}S$ (this gives us a triangular decomposition $S=S^+\cup S^\circ\cup S^-$), and at the end of these cases we show that
$$P=S^+\cup S^\circ=\{\alpha \in S~|~\zeta (\alpha)\geq 0\}.$$
It is a straightforward check that $\Pi:=\bigcup_{i=1}^k\Pi_i\setminus \{\d-\theta_2,\cdots, \d-\theta_k\}$ is a basis for 
${\rm span}_{\Bbb R}S$. Thus, it is enough to define the effect of the functional $\zeta$ on $\Pi$. Before defining $\zeta$, in each of the following cases, we note that $\zeta(\d)>0$ by Remark \ref{3.13}:

{\bf Case 1.} $\d-\theta_1, \cdots, \d-\theta_k\in P\cap -P$: Define 
$$\begin{cases}
    \zeta:{\rm span}_{\Bbb R}S \longrightarrow \Bbb R\\
    \quad \d-\theta_1\mapsto 0\\
    \quad \a_{j,1}\mapsto \dfrac{1}{t_1r_{j,1}},\quad 1\leq j\leq t_1\\
    \quad \a_{j,1} \mapsto 0, \quad t_1+1\leq j\leq n_1\\
\quad \a_{j,i}\mapsto \dfrac{1}{t_ir_{j,i}},\quad 2\leq i\leq k ,~ 1\leq j\leq t_i\\
\quad \a_{j,i}\mapsto 0,\quad 2\leq i\leq k ,~ t_i+1\leq j\leq n_i
\end{cases}$$
Now, for $i\neq 1$, we have 
\begin{align*}
&\zeta(\d-\theta_i)=\zeta(\d)-\zeta(\theta_i)= \zeta(\d-\theta_1)+\zeta(\theta_1)-\zeta(\theta_i)=0+1-1=0.
\end{align*}

{\bf Case 2.} $\d-\theta_1, \cdots, \d-\theta_k\in P\setminus -P$: Define 
$$\begin{cases}
    \zeta:{\rm span}_{\Bbb R}S \longrightarrow \Bbb R\\
    \quad \d-\theta_1\mapsto 1\\
    \quad \a_{j,1}\mapsto \dfrac{1}{t_1r_{j,1}},\quad 1\leq j\leq t_1\\
    \quad \a_{j,1} \mapsto 0, \quad t_1+1\leq j\leq n_1\\
\quad \a_{j,i}\mapsto \dfrac{1}{t_ir_{j,i}},\quad 2\leq i\leq k ,~ 1\leq j\leq t_i\\
\quad \a_{j,i}\mapsto 0,\quad 2\leq i\leq k ,~ t_i+1\leq j\leq n_i
\end{cases}$$
Now, for $i\neq 1$, we have 
\begin{align*}
&\zeta(\d-\theta_i)=\zeta(\d)-\zeta(\theta_i)= \zeta(\d-\theta_1)+\zeta(\theta_1)-\zeta(\theta_i)=1+1-1=0.
\end{align*}

{\bf Case 3.} $\delta-\theta_{1},  \delta-\theta_{l_1}, \cdots, \delta-\theta_{l_p} \in P\setminus-P$, and $\delta-\theta_{i} \in P\cap -P$ for all $i \in\{1, \cdots ,k\}\setminus\{1,l_1,\cdots, l_p\}$: Define 

$$\begin{cases}
    \zeta:{\rm span}_{\Bbb R}S \longrightarrow \Bbb R\\
    \quad \d-\theta_1\mapsto 1\\
    \quad \a_{j,1}\mapsto \dfrac{1}{t_1r_{j,1}},\quad 1\leq j\leq t_1\\
    \quad \a_{j,1} \mapsto 0, \quad t_1+1\leq j\leq n_1\\
\quad \a_{j,i}\mapsto \dfrac{1}{t_ir_{j,i}},\quad i\in\{l_1,\cdots, l_p\} ,~ 1\leq j\leq t_i\\
\quad \a_{j,i}\mapsto 0,\quad i\in\{l_1,\cdots, l_p\} ,~ t_i+1\leq j\leq n_i\\
\quad \a_{j,i}\mapsto \dfrac{2}{t_ir_{j,i}},\quad i\in
\{1,\cdots,k\}\setminus\{1,l_1,\cdots, l_p\} ,~ 1\leq j\leq t_i\\
\quad \a_{j,i}\mapsto 0,\quad i\in
\{1,\cdots,k\}\setminus\{1,l_1,\cdots, l_p\},~ t_i+1\leq j\leq n_i
\end{cases}$$
Now, for $i\in \{l_1,\cdots,l_p\}$, we have 
\begin{align*}
&\zeta(\d-\theta_i)=\zeta(\d)-\zeta(\theta_i)= \zeta(\d-\theta_1)+\zeta(\theta_1)-\zeta(\theta_i)=1+1-1=1
\end{align*}
and for $i\in
\{1,\cdots,k\}\setminus\{l_1,\cdots, l_p\}$
\begin{align*}
&\zeta(\d-\theta_i)=\zeta(\d)-\zeta(\theta_i)= \zeta(\d-\theta_1)+\zeta(\theta_1)-\zeta(\theta_i)=1+1-2=0.
\end{align*}

{\bf Case 4.} $\delta-\theta_{1},  \delta-\theta_{l_1}, \cdots, \delta-\theta_{l_p} \in P\cap -P$, and $\delta-\theta_{i} \in P\setminus -P$ for all $i \in\{1, \cdots ,k\}\setminus\{1,l_1,\cdots, l_p\}$: Define 

$$\begin{cases}
    \zeta:{\rm span}_{\Bbb R}S \longrightarrow \Bbb R\\
    \quad \d-\theta_1\mapsto 0\\
    \quad \a_{j,1}\mapsto \dfrac{2}{t_1r_{j,1}},\quad 1\leq j\leq t_1\\
    \quad \a_{j,1} \mapsto 0, \quad t_1+1\leq j\leq n_1\\
\quad \a_{j,i}\mapsto \dfrac{2}{t_ir_{j,i}},\quad i\in\{l_1,\cdots, l_p\} ,~ 1\leq j\leq t_i\\
\quad \a_{j,i}\mapsto 0,\quad i\in\{l_1,\cdots, l_p\} ,~ t_i+1\leq j\leq n_i\\
\quad \a_{j,i}\mapsto \dfrac{1}{t_ir_{j,i}},\quad i\in
\{1,\cdots,k\}\setminus\{1,l_1,\cdots, l_p\} ,~ 1\leq j\leq t_i\\
\quad \a_{j,i}\mapsto 0,\quad i\in
\{1,\cdots,k\}\setminus\{1,l_1,\cdots, l_p\},~ t_i+1\leq j\leq n_i
\end{cases}$$
Now, for $i\in \{1,l_1,\cdots,l_p\}$, we have 
\begin{align*}
&\zeta(\d-\theta_i)=\zeta(\d)-\zeta(\theta_i)= \zeta(\d-\theta_1)+\zeta(\theta_1)-\zeta(\theta_i)=0+2-2=0
\end{align*}
and for $i\in
\{1,\cdots,k\}\setminus\{l_1,\cdots, l_p\}$
\begin{align*}
&\zeta(\d-\theta_i)=\zeta(\d)-\zeta(\theta_i)= \zeta(\d-\theta_1)+\zeta(\theta_1)-\zeta(\theta_i)=0+2-1=1
\end{align*}
Now, we are ready to show that $P=S^+\cup S^\circ$. Let 
$\a\in S^+\cup S^\circ$. So, we can write $\a$ by some elements of $\Pi$ with positive integers coefficients, because $\Pi$ is a base for $S$.
However, elements of $\Pi$ belong to $P$, and $P$ is closed, which implies that $\a\in P$. Thus $S^+\cup S^\circ\subseteq P$.\\
Conversely, to show $P \subseteq S^+\cup S^\circ$, it is enough to prove that $P\cap S^-=\emptyset$. To see this, assume on the contrary that $P\cap S^-\not=\emptyset$. Then, we pick 
$\a=\sum_{\epsilon \in \Pi}k_\epsilon \epsilon \in P\cap S^-$.
Since $\zeta (\a)<0$, we get that all $k_\epsilon$ are negative integers. Thus, $-\sum k_\epsilon$ is a positive integer. So, we may pick $\a$ such that $-\sum k_\epsilon$ is minimal, i.e. minimal height. Suppose that $-\a\in \Pi$. Then, $-\a=\epsilon$ for some $\epsilon\in \Pi$. As $\zeta(\epsilon)=\zeta(-\a)>0$, we conclude from all the cases defined above for $\zeta$ that 
$\epsilon \in P\setminus -P$. On the other hand, $-\epsilon=\a\in P$, and hence $\epsilon\in -P$. This is a contradiction; consequently, $-\a \not \in \Pi$.
This allows us choosing $\epsilon_0\in \Pi$ such that $\a+\epsilon_0\in S$. It is clear that $\a+\epsilon_0\in S^-\cap P$ and has a smaller height than $\a$, a contradiction. 
Thus $P\cap S^-=\emptyset$ and so
$$P=S^+\cup S^\circ=\{\alpha \in S~|~\zeta (\alpha)\geq 0\}.$$
A straightforward check proves the rest statements.
\end{proof}

It can be seen in Section 2 that the even part of an untwisted affine Lie superalgebra contains two 
or three 
affine Lie subalgebras (in fact, type $D(2,1;\a)^{(1)}$ has three affine Lie subalgebras). We denote by $R(1)$, 
$R(2)$, 
and $R(3)$ the corresponding root system of them.
We set 
$$R(i)^{ln}:=R(i)\cap R^{ln},~~ R(i)^{in}:=R(i)\cap R^{in}.$$ 
Whatever comes for the results and proofs in the following, we consider the following posit:
\begin{equation}\label{convention}
\parbox{5 in}{
{\it We sometimes use the notion $\{i,j,k\}=\{1,2,3\}$.
We assume the status $3$ when the even part of the affine Lie superalgebra has three affine Lie subalgebras.
In the types where the zero part contains two affine Lie subalgebras, 
we can extract the desired one by adapting and reindexing $i$, $j$, and $k$}.}
\end{equation}
Notice that $R(i)\subseteq R_0$, and each $R(i)$  is a symmetric closed subset of $R$, and hence 
$${\mathfrak L}(i):=\bigoplus_{\alpha \in R(i)}\mathfrak L^\alpha$$
is a subsuperalgebra of $\mathfrak L$.

In the rest of this section, we suppose that the $\frak L$-module $M$ is simple and has a finite weight space decomposition with respect to $\frak H$.
We note that $M$ has shadow by Proposition \ref{3.6}. We should also note that in the following proofs, the fact that 
\begin{equation}\label{fact}
    \parbox{5 in}{\it if $\frak k$ is a subalgebra of $\frak L$ with root system $T$, then 
    $T_{re}=T^{ln}\cup T^{in}$, $T^{ln}=\frak B_M\cap T_{re}$ and $R^{in}=\frak C\cap R_{re}$}
\end{equation}
is used.

\begin{Cor}\label{3.18}
Suppose that $i\in \{1,2,3\}$ (recalling (\ref{convention})) and $R(i)=R(i)^+\cup R(i)^0\cup R(i)^-$ is a  triangular decomposition for $R(i)$ with corresponding linear functional $\zeta$ such that 
	$\zeta (\delta)>0$, $R(i)^+\cap R_{re}\subseteq R^{ln}$ and 
	$R(i)^-\cap R_{re}\subseteq R^{in}$. Assume that $W$ is an $\frak L_0$-submodule of $M$. Then there is a positive integer $p$ and $\lambda \in {\rm supp}(W)$ with $(\lambda +\Bbb Z^{>0}p\delta)\cap {\rm supp}(W)=\emptyset$.
\end{Cor}
\begin{proof}
Use (\ref{fact}) and apply Lemma \ref{3.14}.
\end{proof}

\begin{Cor}\label{3.19}
	Suppose that $i,j\in \{1,2,3\}$ and $i\neq j$ (recalling (\ref{convention})).
	Then the following statements hold.
	\begin{itemize}
		\item [(i)] If $R(i)$ 
		is up-nilpotent hybrid, then $R(j)$ is either tight or up-nilpotent hybrid.
		\item [(ii)] If $R(i)$ 
		is down-nilpotent hybrid, then $R(j)$ is either tight or down-nilpotent hybrid.
	\end{itemize}
Moreover, if all $R(i)$ are hybrid, then the fact that one of the $R(i)$s is up-nilpotent (resp. down-nilpotent) causes the others to be up-nilpotent (resp. down-nilpotent). 
\end{Cor}
\begin{proof}
Use (\ref{fact}) and apply Proposition \ref{3.15}.
\end{proof}

We conclude the note with the following observation. 

\begin{Cor}\label{3.20}
	Suppose that  $R(1)$, $R(2)$ and $R(3)$ are hybrid (recalling (\ref{convention})), and $P:=P_1\cup P_2\cup P_3$. Then there is a linear functional $\zeta$ on 
	${\rm span}_{\Bbb R}R_0$ (with $\zeta(\d)>0$) such that 
	$$P=\{\alpha \in R_0~|~\zeta (\alpha)\geq 0\},$$
	in particular, 
	$$\{\alpha\in R_0~|~\zeta (\alpha)>0\}\subseteq R^{ln},~~
	\{\alpha\in R_0~|~\zeta (\alpha)<0\}\subseteq R^{in}.$$	
\end{Cor}

\begin{proof}
Use (\ref{fact}) and apply Proposition \ref{3.17}.
\end{proof}

\end{document}